\documentclass[10pt]{article}

\usepackage{amssymb,amsmath}
\usepackage{amsthm}
\usepackage{amsfonts}
\usepackage{graphicx}
\usepackage{verbatim}
\usepackage[usenames]{color}
\usepackage{dsfont}
\usepackage{amscd,empheq,yhmath}
\usepackage{url}
\usepackage{hyperref}

\newcommand{\R}{\mathbb{R}}
\newcommand{\N}{\mathbb{N}}
\newcommand{\Z}{\mathbb{Z}}

\newcommand{\A}{\mathcal{A}}
\newcommand{\D}{\mathcal{D}}
\newcommand{\Q}{\mathcal{Q}}
\newcommand{\bx}{\bar{x}}
\newcommand{\ba}{\bar{a}}
\newcommand{\bb}{\bar{b}}

\newcommand{\ta}{\tilde{a}}
\newcommand{\tb}{\tilde{b}}
\newcommand{\dagA}{A^\dagger}
\newcommand{\bydef}{\,\stackrel{\mbox{\tiny\textnormal{\raisebox{0ex}[0ex][0ex]{def}}}}{=}\,} 
\newcommand{\hv}{\widehat{v}}
\newcommand{\HF}{{\mathrm{HF}}}

\newtheorem{thm}{Theorem}[section]
\newtheorem{defn}[thm]{Definition}
\newtheorem{cor}[thm]{Corollary}
\newtheorem{lem}[thm]{Lemma}
\newtheorem{prop}[thm]{Proposition}
\newtheorem{rem}[thm]{Remark}

\begin{document}
	
\title{
	Towards computational Morse-Floer homology:\\
	forcing results for connecting orbits\\ 
	by computing relative indices of critical points
}

\author{
Jan Bouwe van den Berg
\thanks{
VU Amsterdam, Department of Mathematics, De Boelelaan 1081, 1081 HV Amsterdam, The Netherlands. ({\tt janbouwe@few.vu.nl})
}
\and 
Marcio Gameiro 
\thanks{
Department of Mathematics, Rutgers University, The State University of New Jersey, Piscataway, NJ, 08854, USA ({\tt gameiro@math.rutgers.edu})
and 
Instituto de Ci\^{e}ncias Matem\'{a}ticas e de Computa\c{c}\~{a}o, Universidade de S\~{a}o Paulo, Caixa Postal 668, 13560-970, S\~{a}o Carlos, SP, Brazil ({\tt gameiro@icmc.usp.br})
}
\and 
Jean-Philippe Lessard
\thanks{
McGill University, Department of Mathematics and Statistics, Burnside Hall, 805 Sherbrooke Street West, Montreal, Qu\'ebec H3A 0B9, Canada. ({\tt  jp.lessard@mcgill.ca})
}
\and 
Rob Van der Vorst
\thanks{
VU Amsterdam, Department of Mathematics, De Boelelaan 1081, 1081 HV Amsterdam, The Netherlands. ({\tt  r.c.a.m.vander.vorst@vu.nl})
}
} 

\date{}

\maketitle

\begin{abstract}
To make progress towards better computability of Morse-Floer homology, and thus
enhance the applicability of Floer theory, it is essential to have tools to
determine the relative index of equilibria. Since even the existence of
nontrivial stationary points is often difficult to accomplish, extracting their
index information is usually out of reach. In this paper we establish a
computer-assisted proof approach to determining relative indices of stationary
states. We introduce the general framework and then focus on three example
problems described by partial differential equations to show how these ideas
work in practice. Based on a rigorous implementation, with accompanying code
made available, we determine the relative indices of many stationary points.
Moreover, we show how forcing results can be then used to prove theorems about
connecting orbits and traveling waves in partial differential equations.
\end{abstract}

\begin{center}
{\bf \small Mathematics Subject Classification (2020)} \\ \vspace{.05cm}
{ \small 57R58 $\cdot$ 35R25 $\cdot$ 65M30 $\cdot$ 65G40 $\cdot$  35K57}
\end{center}

\begin{center}
{\bf \small Keywords} \\ \vspace{.05cm}
 { \small Floer homology $\cdot$ relative indices $\cdot$ ill-posed PDEs $\cdot$ strongly indefinite problems \\ equilibrium solutions $\cdot$ connecting orbits $\cdot$ forcing theorems $\cdot$ computer-assisted proofs}
\end{center}

\begin{center}
{\em Communicated by Peter Bubenik}
\end{center}


\section{Introduction} \label{s:intro}


Morse-Floer homology is a flourishing algebraic-topological construction in the mathematical toolbox for studying variational problems. 
The precursor to Floer homology was first created as a dynamical systems alternative for the intrinsic construction of the homology of compact manifolds --- Morse homology~\cite{MR0228014,MR0029160,MR683171}. Subsequently it has been generalized to many other contexts in analysis and geometry, including infinite dimensional settings.

Without giving a comprehensive description (we refer to~\cite{MR1045282,MR1239174}), we nevertheless first  sketch an outline of Morse-Floer homology  to set the scene. 
A Morse-Floer homology is built from the critical points of the objective functional $\A$, which are described by solutions of the Euler-Lagrange equation $d \A(u)=0$, as well as the non-equilibrium orbits of an associated (formal) gradient flow $u'=-\nabla \A(u)$. The latter depends on the choice of inner product, since it is defined through the relation $d\A(u)v=\langle\nabla \A,v\rangle$ for the directional derivative in the direction~$v$. Besides these two types of solutions (equilibria and connecting orbits) of differential equations, the third ingredient in the homology construction is an (integer valued) index $i(u_*)$ that is associated to each non-degenerate critical point $u_*$. In the classical, finite dimensional setting, this index counts the number $n(u_*)$ of negative eigenvalues of the Hessian matrix $d^2 \!\A (u_*)$, the Morse index, which is interpreted dynamically as the dimension of the unstable manifold of $u_*$ for the gradient flow. 

In an infinite dimensional semi-flow setting the number $n(u_*)$ is usually finite (although the number of positive eigenvalues is infinite now), and one may again use $i(u_*)=n(u_*)$ as the index, and the resulting construction is usually still called a Morse homology.
In other cases, so-called strongly indefinite problems, however, the linear operator associated to the second derivative of $\A$ has both infinitely many negative and infinitely many positive eigenvalues, hence an alternative definition of the index is needed. The breakthrough idea put forward by Floer~\cite{MR965228,MR987770} in
the context of Hamiltonian dynamics, is to define a \emph{relative} index: we have no absolute index but only ever compare the indices of two critical points, say $u_*^1$ and $u_*^2$. In essence this relative index $i(u_*^1,u_*^2)$ counts how many eigenvalues cross over from negative to positive
when one follows a path of linear operators along a homotopy from $d^2 \!\A (u_*^1)$ to $d^2 \! \A (u_*^2)$ --- spectral flow (e.g.\ see~\cite{MR1331677}). Once such a relative index is properly defined (e.g.~taking into account eigenvalues can cross $0$ in both directions), the homology construction can be carried out under suitable compactness and generic transversality conditions. For strongly indefinite problems the resulting homology, based on a relative index rather than a Morse index, is usually called a Floer homology. Clearly Floer homology is a generalization of Morse homology, whereas the Morse index can be seen as a special case of a relative index: $i(u_*^1,u_*^2) = n(u_*^1)-n(u_*^2)$. On the other hand, by choosing a fixed critical point $u_*^0$, one can attach an index to any critical point by setting $\tilde{n}(u_*)=i(u_*^0,u_*)$,
and one may even choose any fixed hyperbolic linear operator as the ``base point'' rather than an equilibrium, see Section~\ref{s:theory} for more details.  We will refer to the collective of such constructions as Morse-Floer theory. 


In this paper we discuss how to compute relative indices of critical points in strongly indefinite variational problems using \emph{computer-assisted} proof techniques. 
For all its popularity and success, Morse-Floer homology is renowned for being difficult to compute explicitly. Recent developments in rigorous computer-assisted analysis of dynamical systems~\cite{notices_jb_jp} brings the opportunity to compute at least some of the ingredients of the Morse-Floer homology construction explicitly, namely critical points and their relative (or Morse) indices. We mention that progress on connecting orbits in infinite dimensions is also being made~\cite{MR3773757,MR2136516,MR3516860}, but the central topic of the current paper is the computations of relative indices.
As hinted at above, this requires the understanding, in a mathematically rigorous computational framework, of the spectral flow along a homotopy between two linear operators on some infinite dimensional space.  
Using relatively simple model equations (so that we can focus on the ideas rather than the technicalities, which may become quite involved in more complex systems), we show that equilibria and, in particular, their relative/Morse indices are computable. We shall also see that the difficulties in computing relative indices are analogous to those for Morse indices (in infinite dimensions, e.g.\ semi-flows). 


Although it is exciting to be able to compute relative indices in view of the long term goal of advancing the applicability of Morse-Floer theory, there are also more direct benefits. In particular, the index information improves the concrete forcing results that we obtain from Morse-Floer homology. 
Indeed, the chain groups in the homology are generated by the critical points, graded by their relative or Morse indices, while
connecting orbits between critical points with relative index $\pm 1$ are used to construct the boundary operator. 
Hence, the homology encodes information about the relation between critical points, their indices, and the connecting orbits between them.  
Exploiting that the homology is an algebraic-topological \emph{invariant} allows us to compare the homology of related systems, for example at different parameter values.  Seeded with computed information about equilibria and their relative indices, this leads to forcing results about the existence of additional critical points as well as connecting orbits. Although some of the forcing results follow from the mere existence of equilibria, the (relative) indices provide refined information, leading to enhanced forcing relations. 
We discuss examples of such forcing results below, where we simultaneously compute and prove equilibria and their relative indices, and then draw conclusions about the minimal set of connecting orbits that is forced by these. 
Even though our choice of model equations are relatively simple so that we can focus on the concepts, some of the forcing results are nevertheless novel. 

The main contribution of this paper is in understanding connecting orbits in PDEs through
a combination of topological methods and computer-assisted proof techniques. We
hasten to add that computer-assisted proofs of equilibrium solutions of PDEs
have a long tradition. The usual approach is based on various variants of
Newton-Kantorovich type arguments. We build on this in the current paper to
determine stationary solutions together with their relative indices, eventually
leading to the forcing results for connecting orbits. For computer-assisted
proofs and rigorous numerics in the context of dynamical systems and
(stationary) PDEs we refer to the expository works
\cite{MR1420838,MR2652784,MR2807595,jay_konstantin_survey,gomez_survey,NPWbook} and the references therein. It is furthermore worth mentioning that
often rigorous control on the spectrum of a certain linear operator (typically
the linearization of the PDE about a numerical solution) leads to a bound on the norm of its inverse and this is then used in a fixed point theorem to prove
(constructively) the existence of a solution, see
e.g.~\cite{MR2019251,MR2492179}. Another link of the computer-assisted proof
literature with our work on relative indices for strongly indefinite problems,
is that when the spectrum is unbounded in one direction only (e.g.\ in parabolic problems), then rigorous computer-assisted
eigenvalue bounds can be used to compute Morse indices, as for example
in~\cite{BHMP}. That powerful methodology, while applicable in a wide variety of contexts (see again~\cite{NPWbook}), does however not extend to the strongly indefinite case considered here.

\subsection*{Three example problems}

To illustrate the central ideas of this paper,
we will use three problems, for which we have implemented the computer-assisted computations to obtain the indices of critical points. 
The first example is the classical application of Floer theory: the Cauchy-Riemann equations
\begin{equation}\label{e:CR}
\begin{cases}
u_t  = v_x + \psi_\lambda(u) ,\\
v_t  = -u_x + v, \\
u_x(t,0)=u_x(t,\pi)=0 , \\
v(t,0)=v(t,\pi)=0 , 
\end{cases}
\end{equation}
where $\psi_\lambda : \R \to \R$ is some smooth nonlinear function.
We write both boundary conditions for $u$ and for $v$, although the Neumann boundary conditions on $u$ are of course imposed automatically by the Dirichlet boundary conditions on $v$ and the second differential equation.
Throughout this paper we will restrict attention to
\[
  \psi_\lambda(u) \bydef \lambda_1 u - \lambda_2 u^3,
\]
where $\lambda_1,\lambda_2 \in \R$ are parameters, but the method works for much more general nonlinearities. Although rescaling could reduce the number of parameters when the signs of $\lambda_1$ and $\lambda_2$ are fixed, keeping two parameters turns out to be advantageous when capitalizing on continuation arguments. In particular, while from the viewpoint of pattern formation and forcing results the interesting case to consider is when both parameters are positive, when homotoping it is convenient to allow $\lambda_1$ to change sign. We come back to this later.

In~\eqref{e:CR} we have chosen Neumann boundary condition on $u$ and Dirichlet boundary conditions on $v$.  In this and all other examples we choose Neumann/Dirichlet boundary conditions rather than periodic ones in order to avoid the issues related to shift invariance (which would make all critical points degenerate).
The time variable is $t$, but this problem is ill-posed hence there is no flow in forward or backward time. 

The equation has a variational structure, as~\eqref{e:CR} is the formal (negative) $L^2$-gradient flow of the action functional
\begin{equation} \label{eq:functional_CR}
\A_{\text{CR}} \bydef \int_0^{\pi} \Bigl[ v u_x - \frac{1}{2} v^2 - \Psi_\lambda(u) \Bigr] dx,
\end{equation}
where $\Psi_\lambda(u)=\frac{\lambda_1}{2}u^2-\frac{\lambda_2}{4}u^4$ is an anti-derivative of $\psi_\lambda(u)$.

Our second example is
\begin{equation}\label{e:TW}
\begin{cases}
  u_{tt} - c u_t + u_{x_1 x_1}+u_{x_2 x_2} + \psi_\lambda(u) =0 \qquad\text{for } x=(x_1,x_2)  \in  [0,\pi] \times [0,\pi],\\
  u_{x_1}(t,0,x_2)=u_{x_1}(t,\pi,x_2)=0,\\
      u_{x_2}(t,x_1,0)=u_{x_2}(t,x_1,\pi)=0.
\end{cases}
\end{equation}
Here $c > 0$ is a parameter that has the interpretation of the wave speed, since~\eqref{e:TW} results from substituting a travelling wave Ansatz into the parabolic equation
\begin{equation}\label{e:cylinder}
  u_t = \Delta u +\psi_\lambda(u) = u_{x_1 x_1} + u_{x_2 x_2}  + u_{x_3 x_3} +\psi_\lambda(u), \qquad\text{for } t,x_3 \in \R, \, x_1,x_2 \in  [0,\pi],
\end{equation}
with Neumann boundary conditions on the ``cylindrical'' spatial domain $[0,\pi]^2 \times \R$.
Hence solutions of~\eqref{e:TW} correspond to travelling wave solutions of~\eqref{e:cylinder} on the infinite cylinder, see e.g.~\cite{BakkervdBergvdVorst,FSV,Gardner,Mielke}.
The problem~\eqref{e:TW} is not quite a gradient flow, but rather it is gradient-like. This still suffices for a Morse-Floer homology construction. Indeed, for the problem~\eqref{e:TW} the details of this construction can be found in~\cite{BakkervdBergvdVorst}. The functional
\begin{equation} \label{eq:functional_TW}
\A_{\text{TW}} \bydef \int_0^\pi \int_0^\pi \Bigl[ -\frac{1}{2} (u_t)^2 + \frac{1}{2} ((u_{x_1})^2+(u_{x_2})^2) - \Psi_\lambda(u) \Bigr]  dx_1 dx_2
\end{equation}
serves as Lyapunov function for solutions of~\eqref{e:TW} for any $c>0$.

Our third example is the Ohta-Kawasaki equation~\cite{Ohta-Kawasaki}
\begin{equation}\label{e:OK}
	\begin{cases} 
   u_t = - u_{xxxx} - (\psi_\lambda(u))_{xx} - \lambda_3 u, &\quad\text{for } x\in [0,\pi],\\
    u_x(t,0)=u_x(t,\pi)=0, \\
    u_{xxx}(t,0)=u_{xxx}(t,\pi)=0,\\
	\int_0^\pi u(0,x) dx =0 ,
  \end{cases}
\end{equation}
which is used to model diblock copolymers~\cite{MR1334695,MR2496714,MR2685742}.
The extra parameter $\lambda_3 \geq 0$ describes the strength of the (attractive) long range interactions in the mixture.
The space of functions $u$ satisfying $\int_0^\pi u(x) dx =0$ is invariant (the general Ohta-Kawasaki model has a parameter $m$ that denotes the mass ratio of the two constituents in the mixture; for simplicity we consider the case $m=0$ only, corresponding to a 50\%-50\% mixture).  
Equation~\eqref{e:OK} does not have an ill-posed initial value problem, but generates a semi-flow. 
Indeed, we use it to illustrate that the computation of a Morse and a relative index can be treated in a unified framework. 
The flow generated by~\eqref{e:OK} is the formal negative gradient flow in the dual space of $\{u\in H^1 : \int_0^\pi u(x)dx =0\}$ (a natural space for modeling mass conservation) for the functional
\begin{equation} \label{eq:functional_OK}
\A_{\text{OK}} \bydef \int_0^\pi \Bigl[ \frac{1}{2} (u_x) ^2 - \Psi_\lambda(u) + \frac{\lambda_3}{2} (\phi_x)^2  \Bigr] dx,
\end{equation}
where $\phi$ is the unique solution of the elliptic problem
\[
\begin{cases}
    -\phi_{xx} = u, & \quad\text{for } x\in [0,\pi], \\
    \phi_x(0)=\phi_x(\pi)=0,\\
	\int_0^\pi \phi(x) dx =0 .   
\end{cases}
\]

\subsection*{Sample results}

In gradient(-like) systems the only type of (bounded) solutions that exist for all time $t \in \R$ are equilibria and heteroclinic connections. We will not assume all the equilibria to be nondegenerate (since that is very difficult to check). Therefore, we define connecting orbits as orbits for which the $\alpha$ and $\omega$ limit sets are disjoint and consists of equilibria only. 
Although generically these are classical connecting orbits between nondegenerate equilibria (indeed, the definition of Morse-Floer homology is built on that), this broader definition allows one to draw more general conclusions.

For the Cauchy-Riemann problem~\eqref{e:CR} we determine the Floer homology by continuation to the case $\widetilde{\psi}(u)=-u-u^3$. In that case there is a unique equilibrium solution $(u,v)(x) \equiv (0,0)$. This stationary point is hyperbolic and we use the associated linear operator as the base point relative to which we define indices. We note that $(u,v)=(0,0)$ is an equilibrium for any $\lambda_1 , \lambda_2 \in \R$, but we choose a linear operator as the base point for the relative index that is \emph{independent} of $\lambda$ and thus in general does \emph{not} correspond to the linearization at the trivial equilibrium. Indeed, choosing a fixed base point allows one to use the invariance of Floer homology under continuation (see Theorem~\ref{homotopy2}) to determine explicitly the (Betti numbers of the) Floer homology, see the proof below.

We now state a sample result. Of the seven equilibria at these particular parameter values only two need to be verified by computer-assisted proof, as the remainig ones are either related by symmetry or homogenous.
\begin{thm}\label{t:CR}
	For $\lambda_1=\lambda_2=6$ the Cauchy-Riemann problem~\eqref{e:CR} has at least seven equilibrium solutions with relative indices $0,0,1,1,2,2,3$.  Moreover, there are at least three connecting orbits, of which at least two have nontrivial spatial dependence.
\end{thm}
\begin{proof}[Outline of proof]
Continuation of the nonlinearity $\psi_\lambda$ for $\lambda=(6,6)$ to the base point at $\lambda=(-1,1)$ can be performed within the class of coercive nonlinearities, see Section~\ref{geninv12}. This guarantees the necessary compactness properties,see Proposition~\ref{compact1}. 
We obtain $\beta_0=1$ and $\beta_k=0$ for $k\neq 0$, where $\beta_k$ are the Betti numbers of the Floer homology $\HF_k\bigl(\mathcal{S^\infty},\psi\bigr)$, 
where $\mathcal{S^\infty}$ is maximal invariant set in $\mathcal{N} = C^1([0,\pi])$.
cf. Section \ref{construction1}.  

At $\lambda=(6,6)$ the indices of the homogeneous equilibria $(u,v)=(1,0)$,  $(u,v)=(0,0)$ and $(u,v)=(-1,0)$ are $0$,~$3$ and~$0$, respectively, as can be verified by hand or computer.
Two of the other equilibria are depicted in Figure~\ref{f:CR}; see Section~\ref{sec:rig_comp_critical points} for an explanation about the rigorous error control on the distance between the graphs depicted and the true solutions. Their relative indices are 1 and 2. The remaining two equilibria are related to these via the transformation $(u,v) \mapsto (-u,-v)$.

The results on the number and type of connecting orbits are a consequence of the forcing Lemma~\ref{l:forcing}. There, the integers $\zeta_k$ are introduced as lower bounds on the number of hyperbolic critical points of relative index $k$. The relative index information on the seven equilibria implies that we may set
\[
  \zeta_0=2, \qquad
   \zeta_1=2, \qquad
    \zeta_2=2, \qquad
	 \zeta_3=1.
\]
On the other hand, as mentioned above (when we chose the base point) the only nonzero Betti number is $\beta_0=1$. 
The multiplicity result then follows directly from the forcing Lemma~\ref{l:forcing}. Additionally it implies that each of the four nonhomogeneous equilibria (the ones with relative index $1$ and~$2$) forms the $\alpha$ or $\omega$ limit set of at least one connecting orbit.

The remaining details of the proof are filled in Sections~\ref{s:setuphomotopy}  (existence theorem for equilibria and computation of the relative indices) and~\ref{s:CR} (bounds needed for the computer-assisted part of the proof).
\end{proof}

We note that large parts of the analysis of~\eqref{e:CR} can be done by hand, since the equilibria for the particular choice of the right-hand side coincide with those of the \emph{Allen-Cahn} or \emph{Chaffee-Infante} parabolic problem
\[
  u_t= u_{xx} + \psi_{\lambda}(u).
\]
This (bifurcation) problem is analyzed in detail in~\cite{MR804887,MR1347417}.
Furthermore, by using the symmetry one could obtain somewhat stronger forcing results, but we do not pursue that here as it is beside the point of this paper.

For the other two examples we obtain similar results, but here no alternative using pencil-and-paper analysis is available.
For the problem~\eqref{e:TW} we again first select a base point, relative to which we define the indices. Namely, as for the problem~\eqref{e:CR}, for the linear case $\widetilde{\psi}(u)=-u$ there is a unique, hyperbolic equilibrium solution $u(x) \equiv 0$. We choose the associated linear operator (where we may pick any $c>0$) as our base point. We can now formulate the following sample results. 
\begin{thm}\label{t:TW}
	For $\lambda_1=\lambda_2=12$ the travelling wave problem~\eqref{e:TW} has at least 71 equilibrium solutions with relative indices $0$ (2$\times$), $2$ (8$\times$), $3$ (8$\times$), $4$ (8$\times$), $5$ (8$\times$), $6$ (12$\times$), $7$ (8$\times$), $8$ (6$\times$), $10$ (4$\times$), $11$ (4$\times$), $12$ (2$\times$),  and $13$ (1$\times$). Moreover, for any $c>0$ there are at least 35 connecting orbits, corresponding to travelling waves of~\eqref{e:cylinder}. Each of the 68 nonhomogeneous equilibria is the $\alpha$ or $\omega$ limit set of a connecting orbit.
\end{thm}
The nonhomogeneous equilibria are depicted in Figure~\ref{f:TW}. The problem allows a symmetry group of order $16$, generated by the operations 
\[ 
   x_1 \mapsto \pi-x_1 \qquad  x_2 \mapsto \pi-x_2
   \qquad (x_1,x_2) \mapsto (x_2,x_1) \qquad  u \mapsto -u.
\]
For each equilibrium represented in Figure~\ref{f:TW} there are additional ones generated by these operations (the orbit under the action of the symmetry group). The number of such symmetry-related equilibria is indicated in Figure~\ref{f:TW}. We note that the symmetry slightly reduces the computational cost. We regard this aspect as somewhat immaterial, as the computer-assisted proofs of the equilibria generalize to rectangles in a straighforward manner. The proof of Theorem~\ref{t:TW} is essentially the same as the one of Theorem~\ref{t:CR}. Indeed, once the Floer homology construction has been justified, the forcing lemma~\ref{l:forcing} again applies, with relatively minor  differences in the estimates and computational details for the equilibria, which now depend on two spatial variables. The construction of Floer homology theory in this case is less classical, but has been accomplished in~\cite{BakkervdBergvdVorst}. The result in Theorem~\ref{t:TW} complements the ones obtained in~\cite{FSV}, where a result similar to Lemma~\ref{l:forcing} is proven using the Conley index, but without the information on the existence of equilibria provided by our computer-assisted approach.

For the problem~\eqref{e:OK} choosing a base point is not an issue. Since the problem is not ill-posed, one may just use the classical Morse index, and Floer thoery is replaced by infinite dimensional Morse theory, which is standard in the context of parabolic PDEs of gradient type, cf.~\cite{Weber1,Weber2}.
\begin{thm}\label{t:OK}
	For $\lambda_1=\lambda_2=9$ and $\lambda_3=4.5$ the Ohta-Kawasaki problem~\eqref{e:OK} has at least $9$ equilibrium solutions with Morse indices $0$ (4$\times$), $1$ (4$\times$) and $2$ (1$\times$). Moreover, there are at least 4 connecting orbits, each having nontrivial spatial dependence.
\end{thm}
The nontrivial equilibria are depicted in Figure~\ref{f:OK}.
The proof is analogous to those discussed above, with some computational details for this particular problem provided in Section~\ref{s:OK}.
This complements results from \cite{MR2136516}, where constructive computer-assisted proofs of existence of connecting orbits in Ohta-Kawasaki are obtained. 

\begin{figure}[t]
\centerline{
\includegraphics[width=0.4\textwidth]{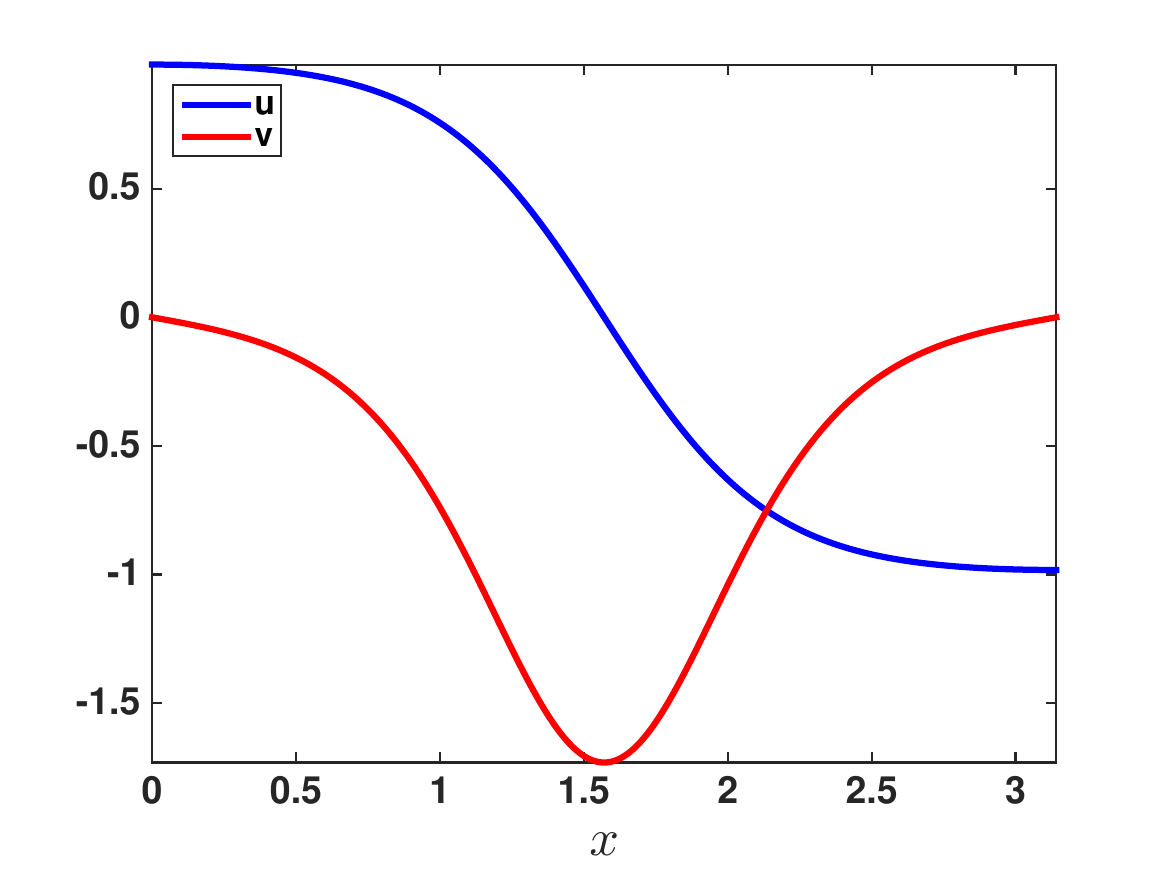} ~~
\includegraphics[width=0.4\textwidth]{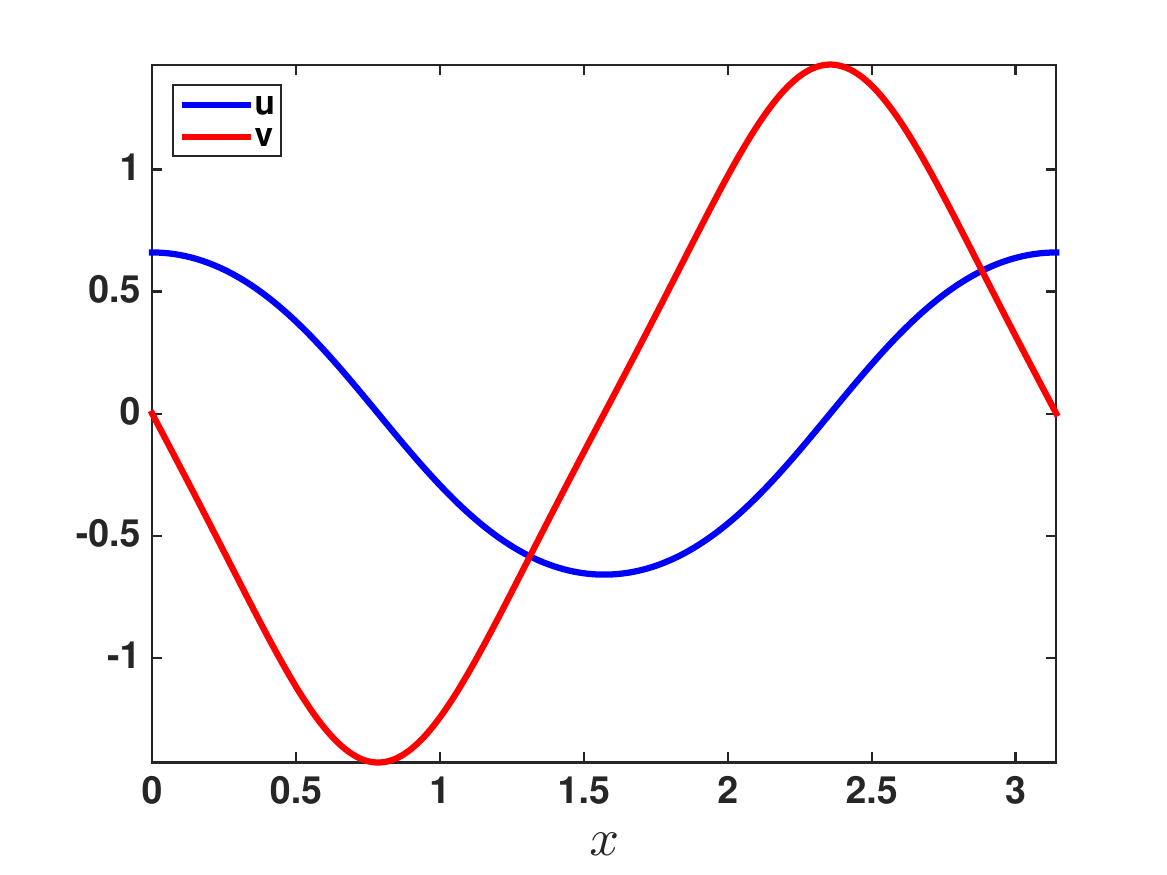}
}
\caption{Equilibrium solutions of~\eqref{e:CR} for $\lambda_1=\lambda_2=6$. The error in the plots in the $C^0$ norm is no more than $5 \cdot 10^{-11}$. To each equilibrium $(u(x),v(x))$ corresponds another equilibrium $(-u(x),-v(x))$ with the same relative index. Moreover, the homogeneous states $(-1,0)$, $(0,0)$ and $(1,0)$ are also equilibria. The states $(\pm 1,0)$ have relative index $0$, the equilibrium on the left has index $1$, the one on the right has index $2$, and the state $(0,0)$ has relative index $3$.}
\label{f:CR}
\end{figure}
\begin{figure}[p]
\vspace*{-17.1pt}
\centerline{\includegraphics[width=0.33\textwidth]{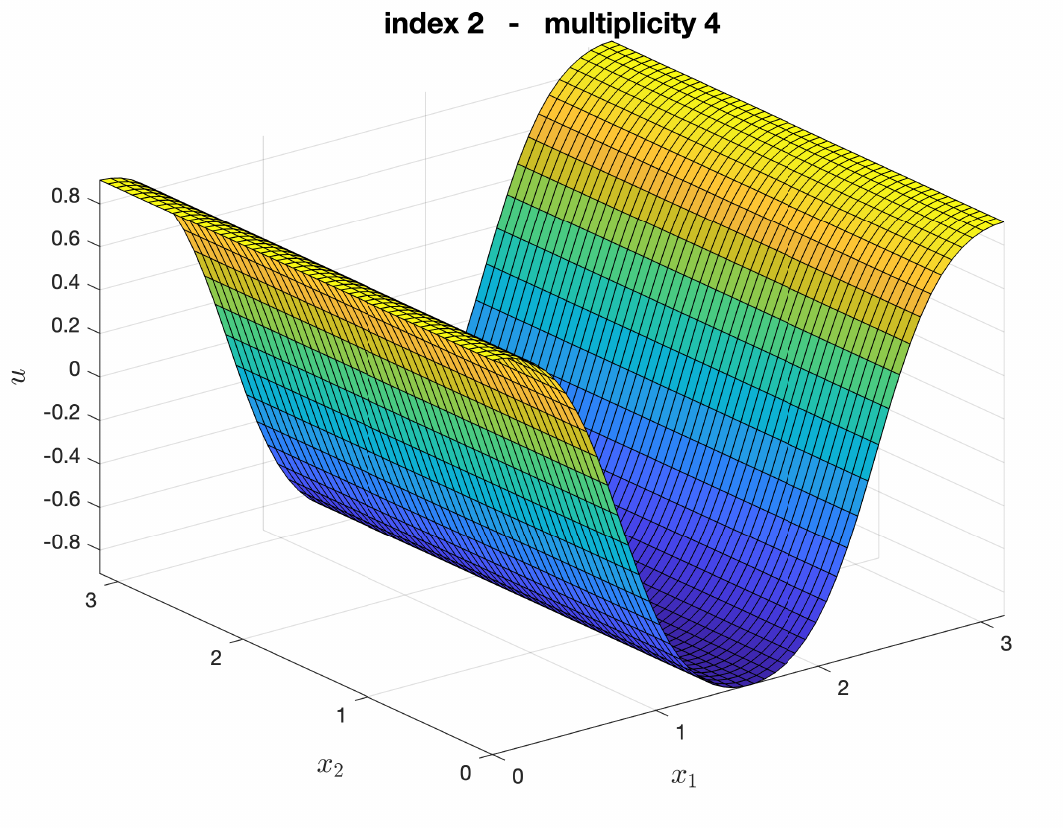}\includegraphics[width=0.33\textwidth]{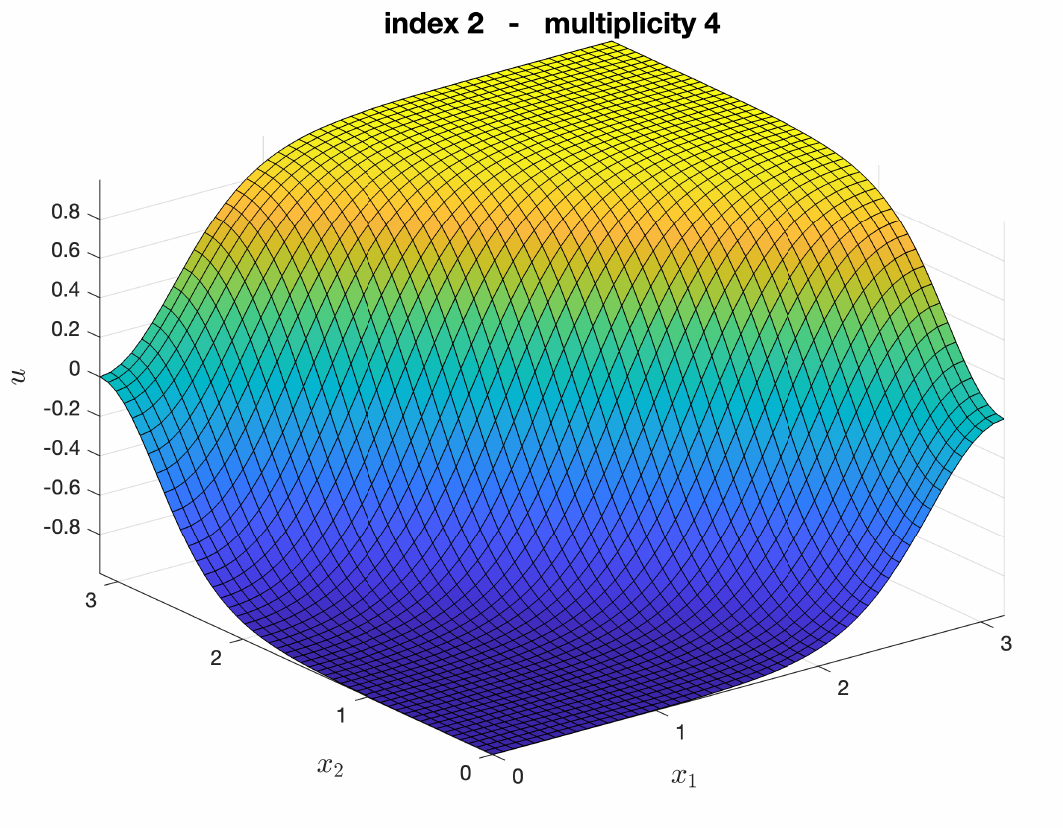}\includegraphics[width=0.33\textwidth]{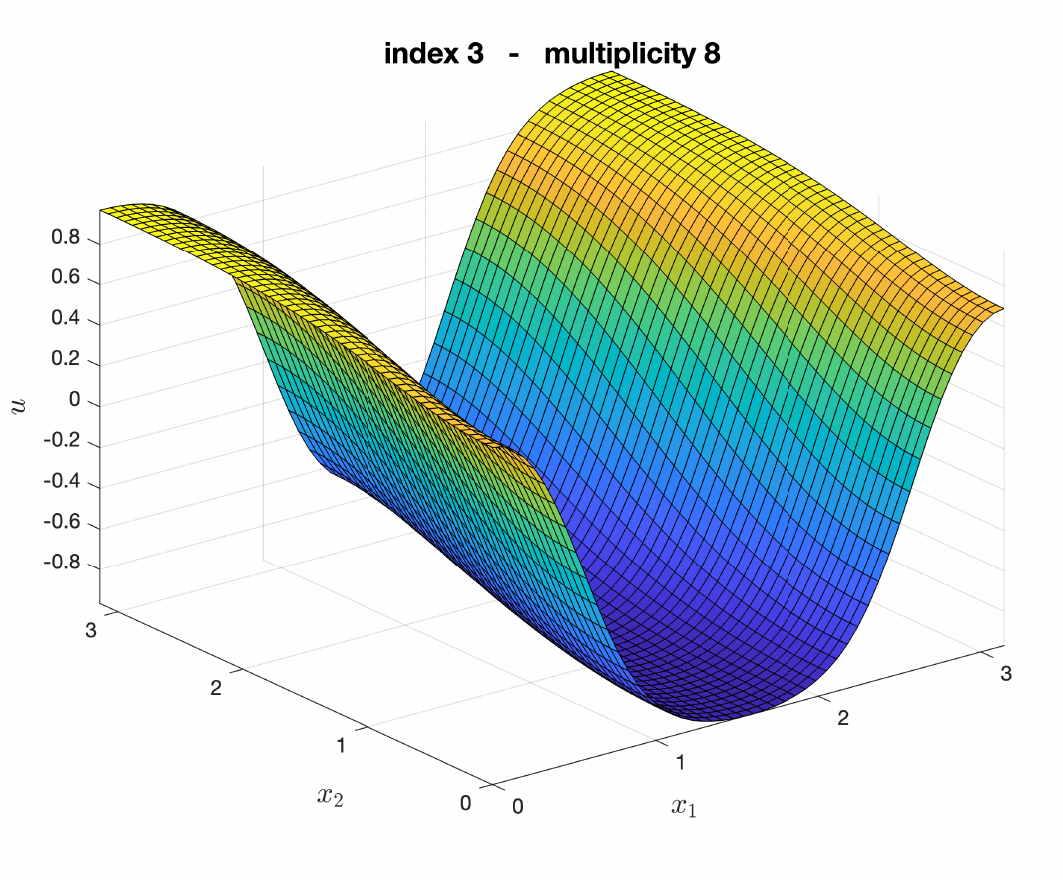}}
\centerline{\includegraphics[width=0.33\textwidth]{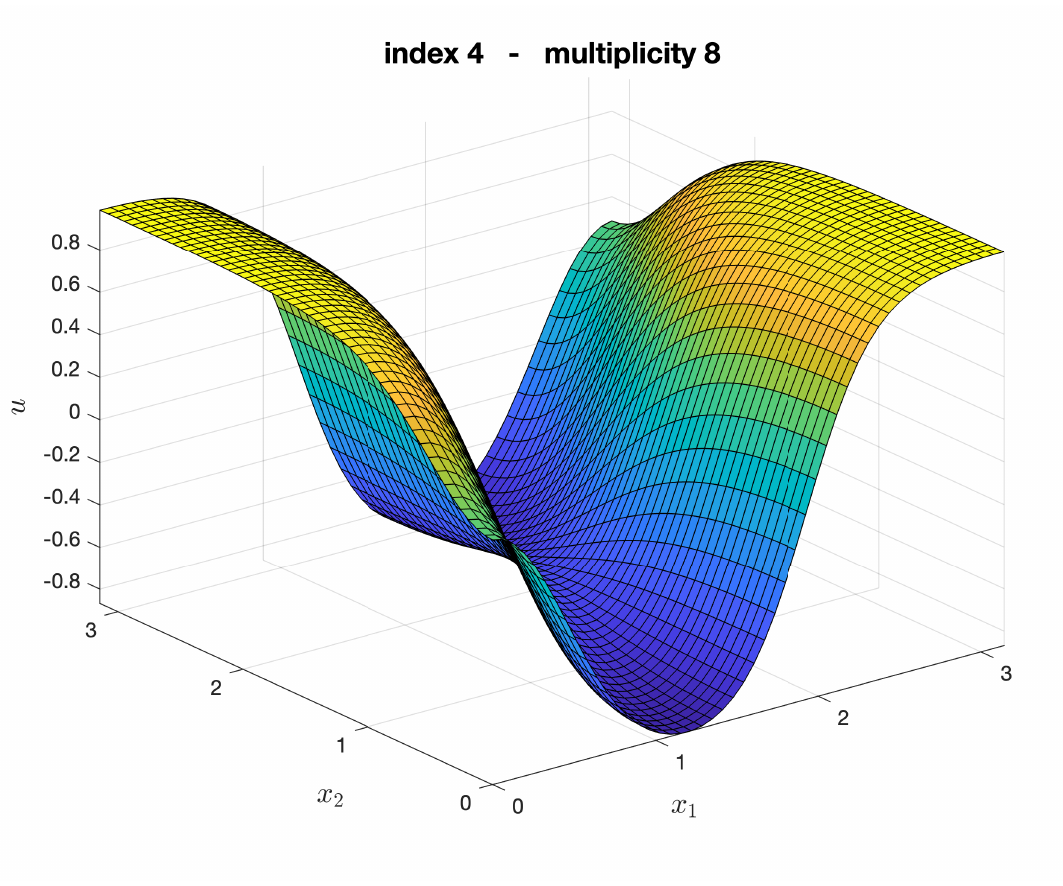}\includegraphics[width=0.33\textwidth]{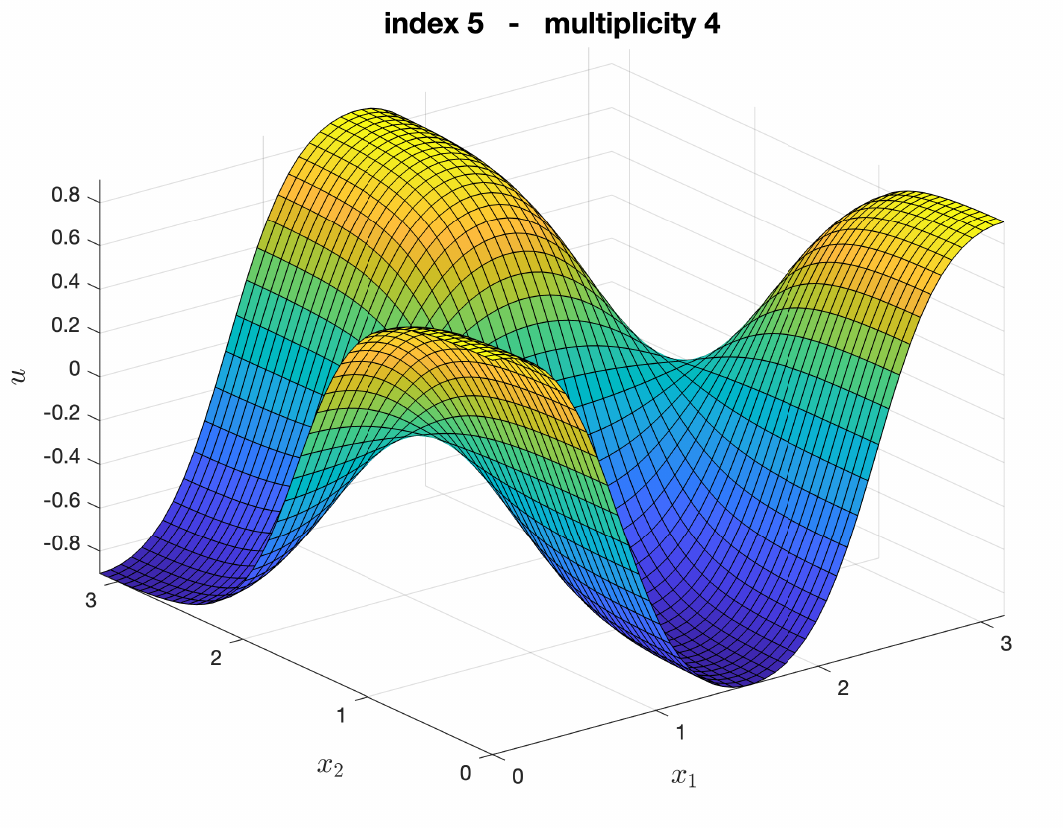}\includegraphics[width=0.33\textwidth]{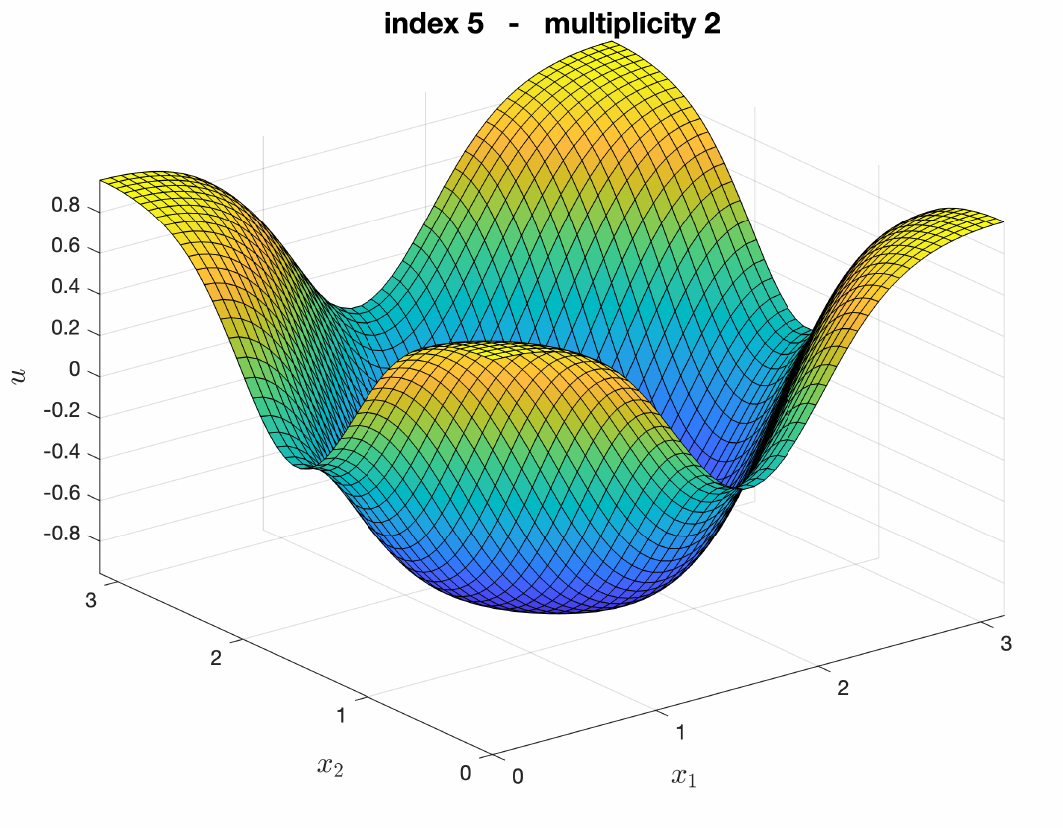}}
\centerline{\includegraphics[width=0.33\textwidth]{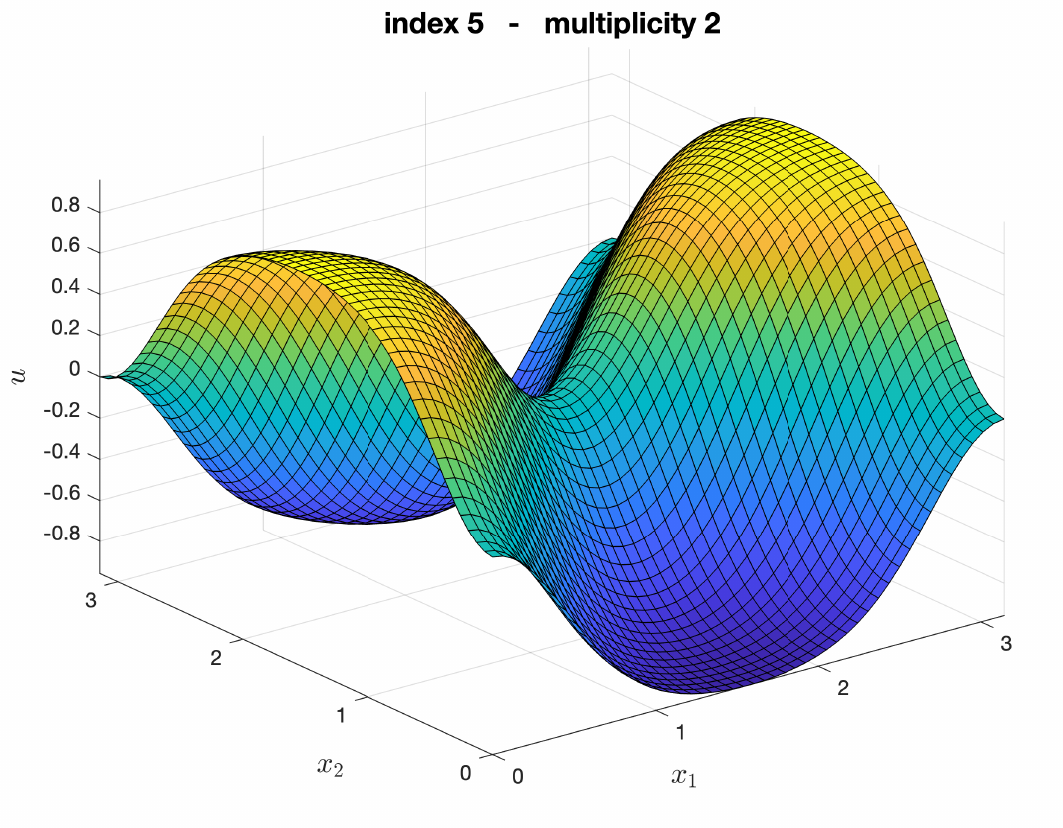}\includegraphics[width=0.33\textwidth]{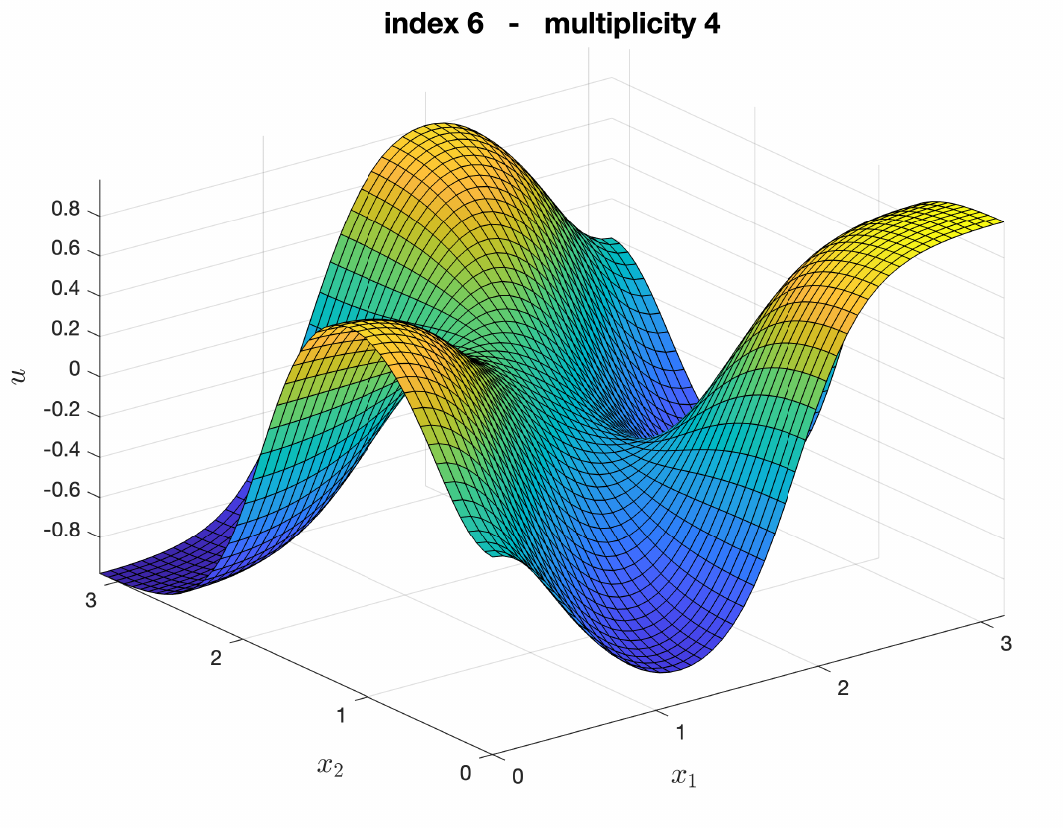}\includegraphics[width=0.33\textwidth]{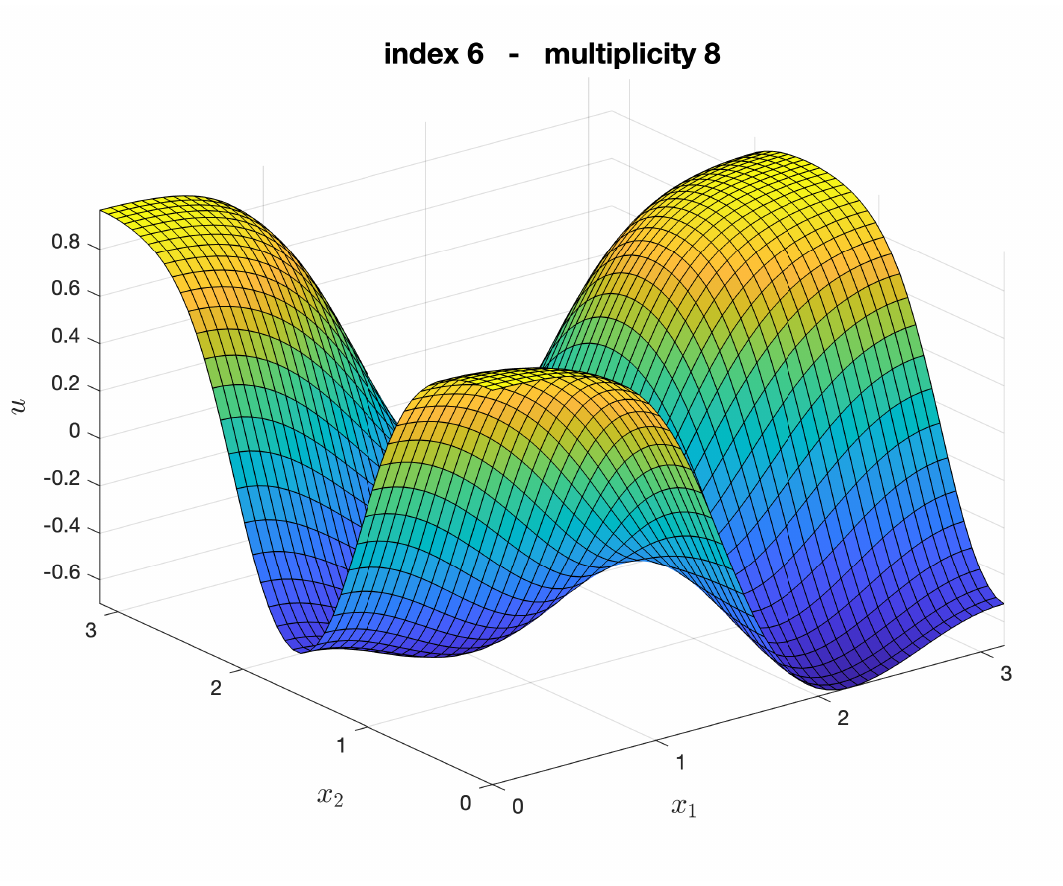}}
\centerline{\includegraphics[width=0.33\textwidth]{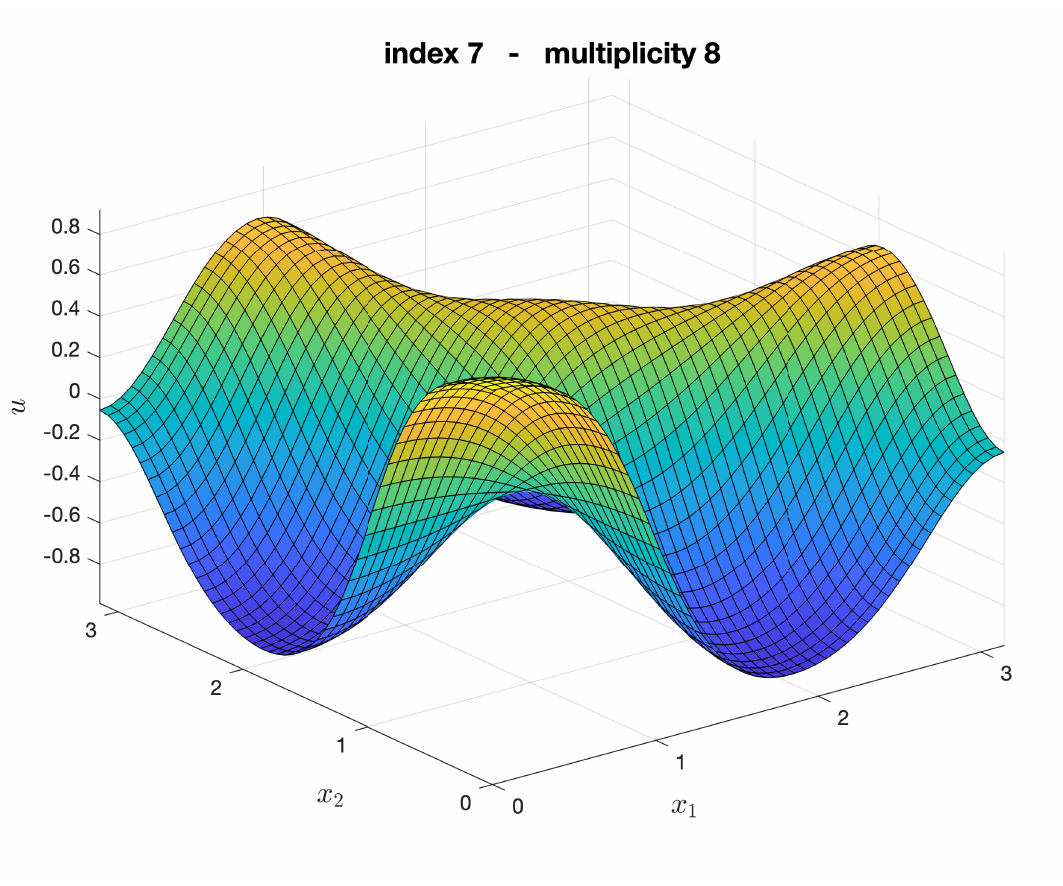}\includegraphics[width=0.33\textwidth]{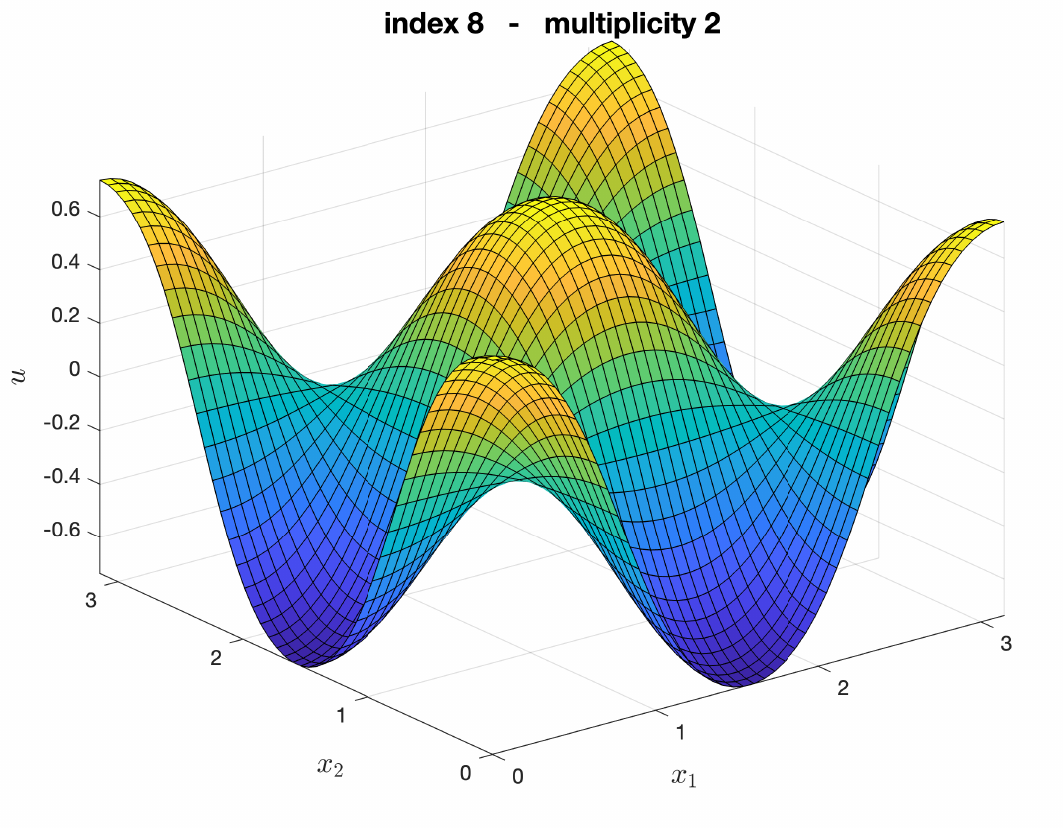}\includegraphics[width=0.33\textwidth]{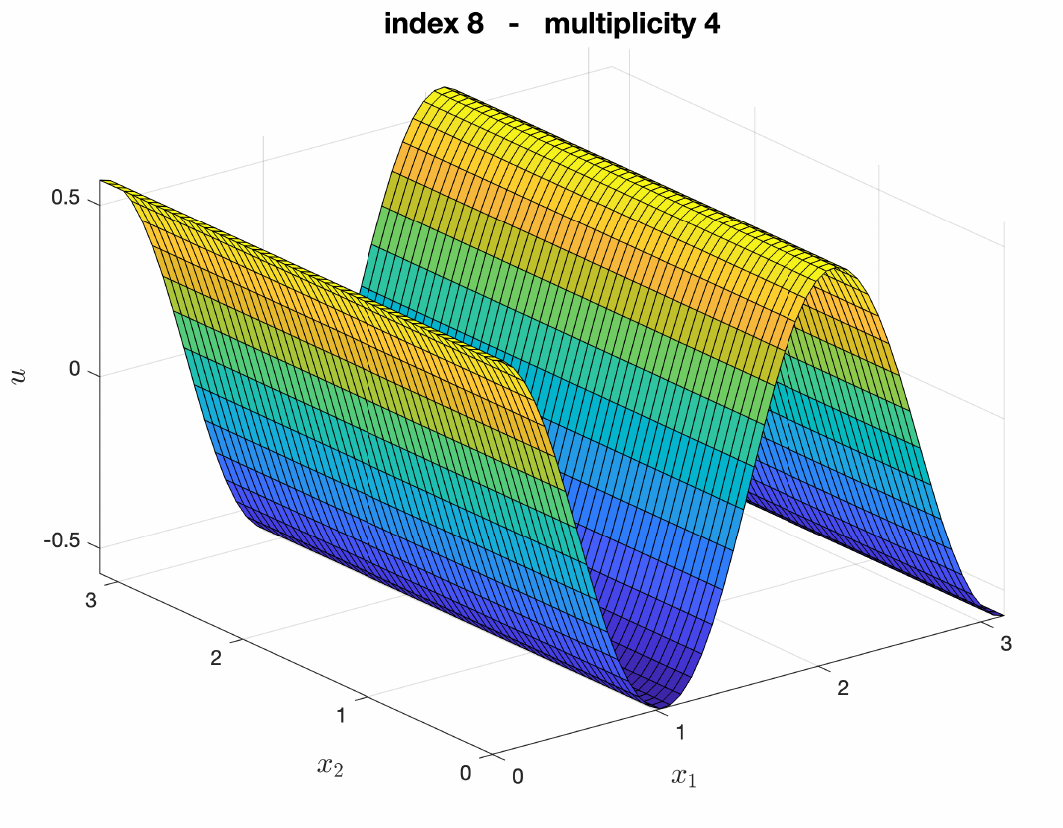}}
\centerline{\includegraphics[width=0.33\textwidth]{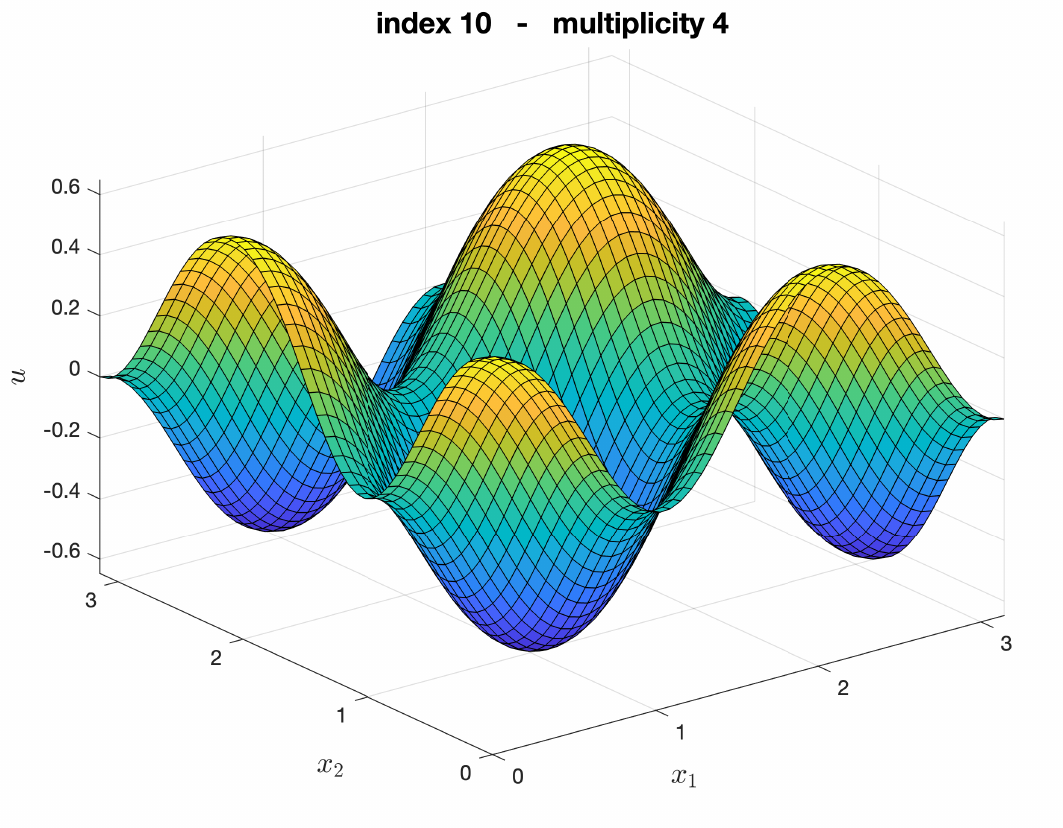}\includegraphics[width=0.33\textwidth]{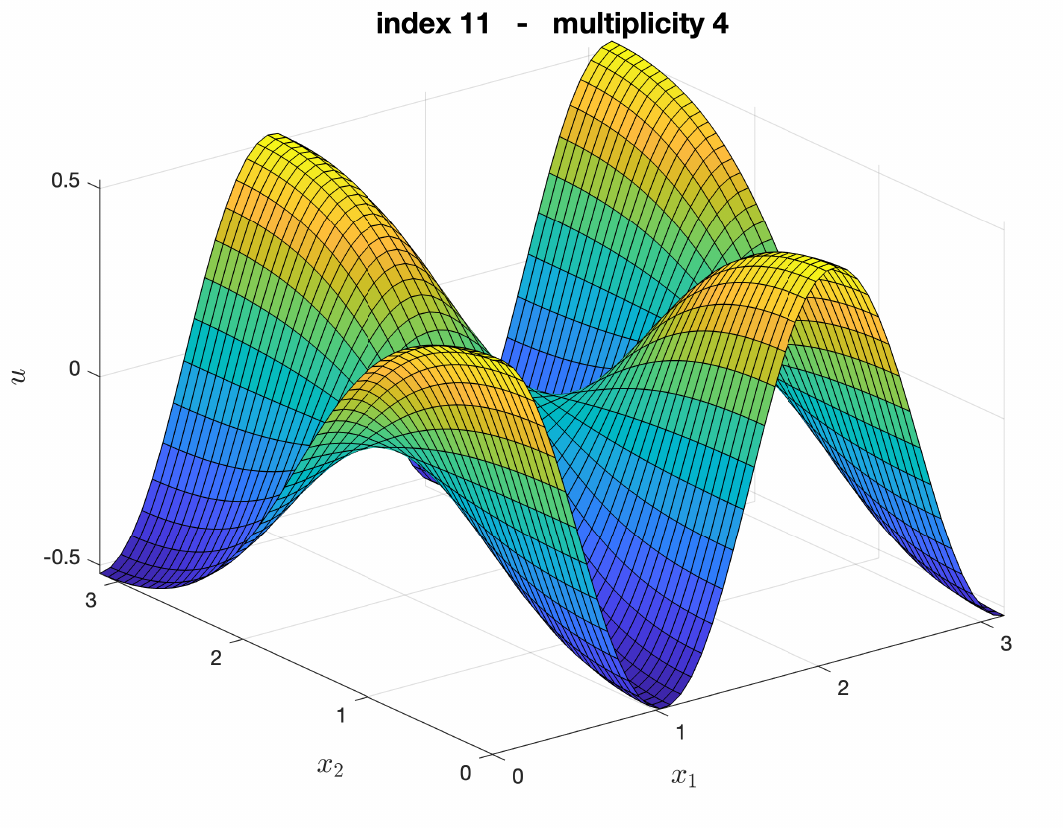}\includegraphics[width=0.33\textwidth]{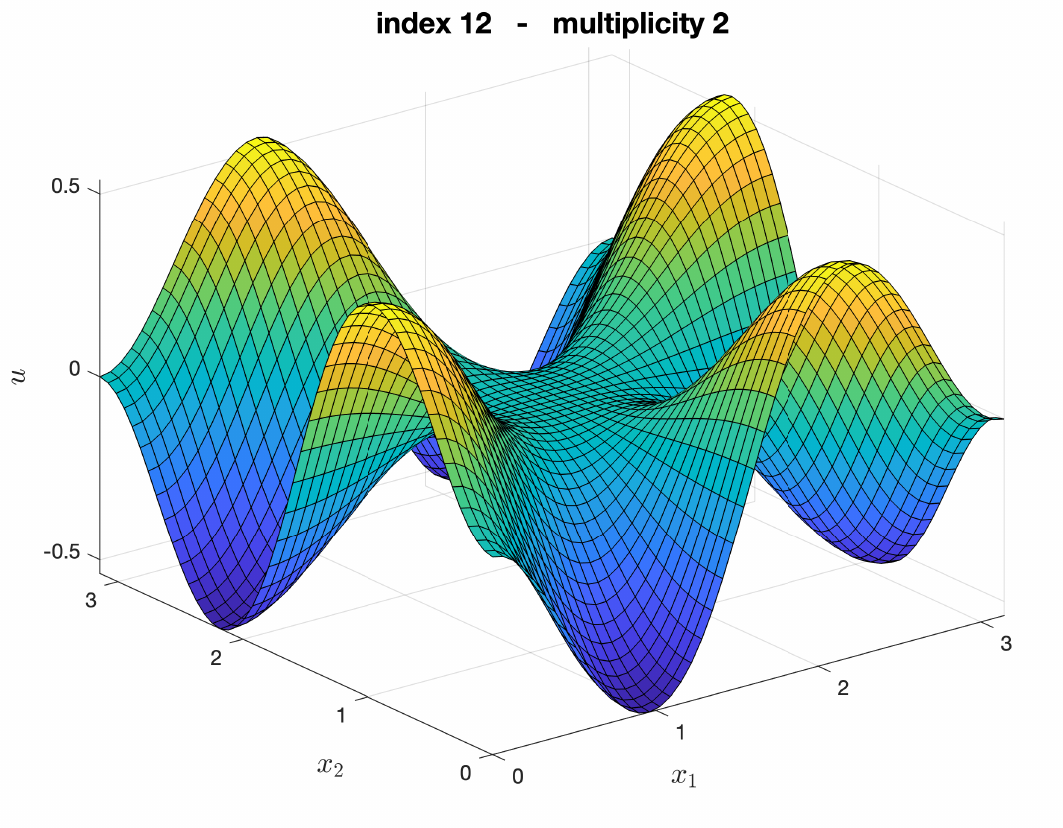}}
\caption{Nonhomogenous equilibrium solutions of~\eqref{e:TW} for $\lambda_1=\lambda_2=12$. The relative index is indicated above each graph. The error in the plots in the $C^0$ norm is less than $3\cdot10^{-5}$. Additionally, the homogenous solutions $u \equiv \pm 1$ and $u \equiv 0$ have indices 0 and 13, respectively. The multiplicity mentioned above each graph is the number of symmetry-related equilibria, as explained in the main text.}
\label{f:TW}
\end{figure}

\begin{figure}[t]
\centerline{
\includegraphics[width=0.4\textwidth]{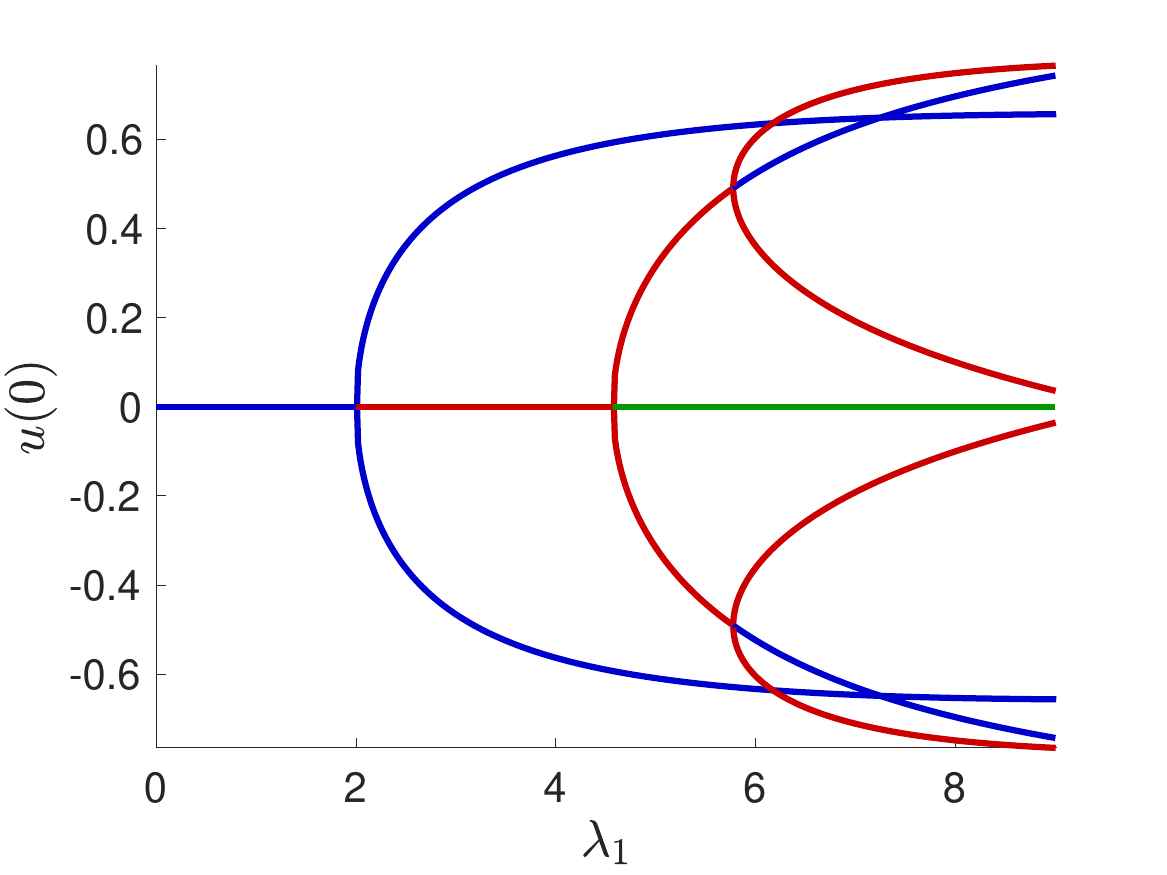}
\includegraphics[width=0.4\textwidth]{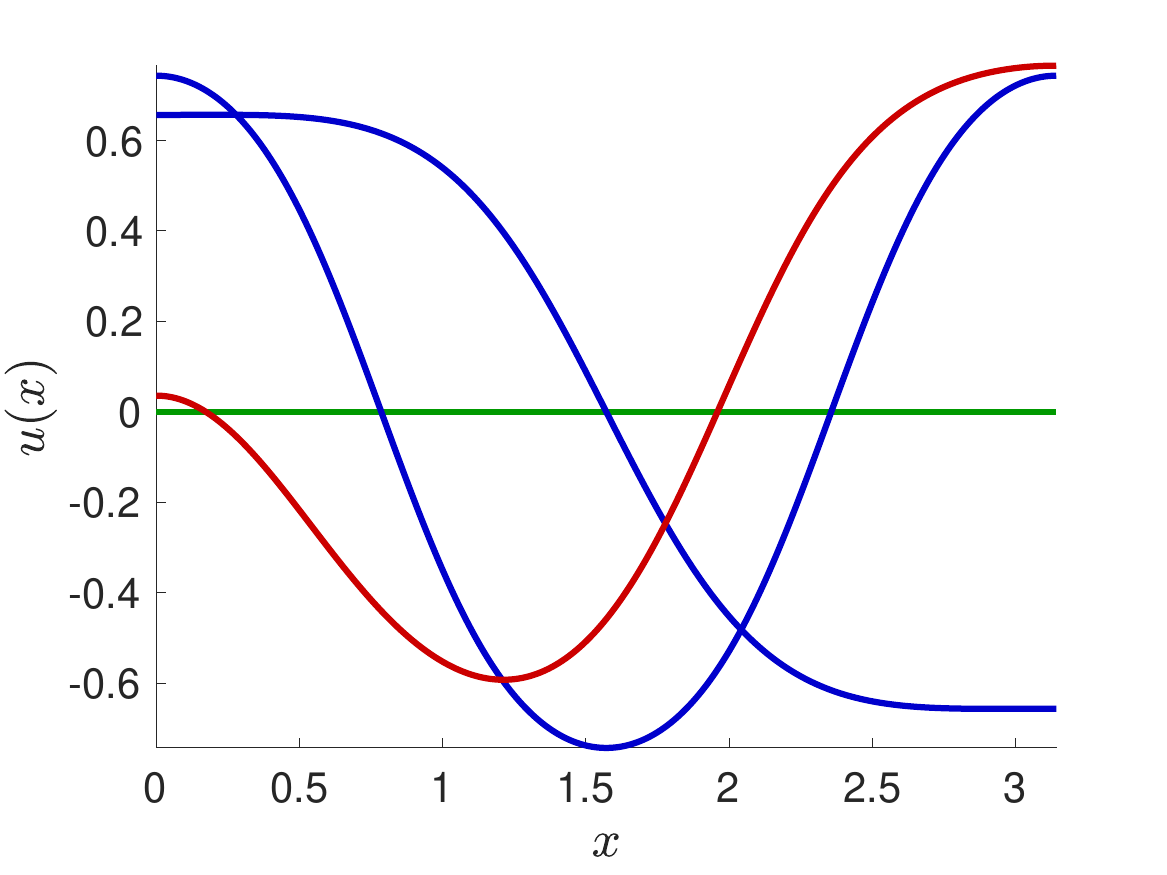}
}
\caption{(Left) Bifurcation diagram of equilibria of \eqref{e:OK} when $\lambda_1=\lambda_2$ varies over the interval $[0,9]$ with Morse indices $0$ (blue), $1$ (red) and $2$ (green). (Right) Equilibrium solutions of~\eqref{e:OK} for $\lambda_1=\lambda_2=9$ and $\lambda_3=4.5$. The trivial solution (in green) has index $2$. The equilibria in blue have index $0$ while the one in red has index $1$. To each blue equilibrium solution $u(x)$ corresponds a solution $-u(x)$ having the same index. Moreover, to the red solution $u(x)$ corresponds the three other equilibria $-u(x)$ and $\pm u(\pi-x)$. The error in the plots in the $C^0$ norm is no more than $6 \cdot 10^{-14}$.}
\label{f:OK}
\end{figure}

\subsection*{Outline of the paper}

The outline of this papers is as follows.
In Section~\ref{s:theory} we give a concise outline of the construction  of Morse-Conley-Floer homology, and we discuss the forcing relation in Morse-Conley-Floer theory between critical points (with their indices) and connecting orbits.
In Section~\ref{s:setuphomotopy} we introduce the computational setup for computing-proving the equilibria and their (relative, Morse) indices, as well as the spectral-flow and homotopy arguments that turn the computational results into rigorous ones. 
We use that in a Fourier series setting, which the three problems~\eqref{e:CR}, \eqref{e:TW} and~\eqref{e:OK} all fit in, the spectral flow properties that we need are particularly convenient from a computational point of view. This is due to the lack of explicit boundary conditions (which are absorbed in the Banach spaces we choose to work in) and the fact that the dominant differential operators are diagonal in the Fourier basis.
In Sections~\ref{s:CR},~\ref{s:TW} and~\ref{s:OK} we provide computational details for each of the example problems~\eqref{e:CR}, \eqref{e:TW} and~\eqref{e:OK}, respectively. Since the computational particulars are not the core of the present paper, and many of the estimates are available elsewhere, we keep those sections brief, using citations to the literature where appropriate. All computer-assisted parts of the proofs are performed with code available at~\cite{codes_webpage}.


\section{Morse-Conley-Floer homology} \label{s:theory}

In order to establish connecting orbits in various classes of partial differential equations, including strongly indefinite ones, we want to use a topological-algebraic invariant.
Since the systems under study are either gradient or gradient-like systems, a natural choice is to use an intrinsically  defined invariant such as a Morse homology or Floer homology.
For definiteness, to define an appropriate index theory we focus on the Cauchy-Riemann equations in~$\R^2$. We emphasize that the same methods apply more generally, in particular to the other problems introduced in Section~\ref{s:intro}.

Consider equations of the form
\begin{equation}\label{eq:CR_BVP1}
\left\{
\begin{aligned}
u_t  &= v_x + \psi_\lambda(u) \\
v_t  &= -u_x + v, \\
\end{aligned}
\right.
\end{equation}
where $z=(u,v)\colon \R\times [0,\pi] \to \R^2$, with the boundary conditions $u_x(0)=u_x(\pi)=0$ and
$v(0)=v(\pi)=0$.
The above equations are the negative $L^2$-gradient flow of the functional 
\begin{equation*}
\A^\epsilon_{\text{CR}}(z) \bydef \int_0^{\pi} \Bigl[ v u_x - \frac{1}{2} v^2 - \Psi_\lambda(u)-\epsilon h(x,u,v) \Bigr] dx,
\end{equation*}
when $\epsilon=0$.
The values $\epsilon>0$ are referred to as the perturbed problem.
For simplicity, suppose that $\Psi_\lambda$ is a superquadratic polynomial in $u$.
For the function $h$ we assume throughout that $|h_z(x,u,v)|\le o(|u|+|v|)$ uniformly in $x\in [0,\pi]$.
The perturbed equations are 
\begin{equation}\label{eq:CR_BVP1e}
\left\{
\begin{aligned}
u_t  &= v_x + \psi_\lambda(u) +\epsilon h_u(x,u,v), \\
v_t  &= -u_x + v + \epsilon h_v(x,u,v), \\
\end{aligned}
\right.
\end{equation}
with boundary conditions $v(0)=v(\pi)=0$. In the variational formulation of $\A^\epsilon_{\text{CR}}$ the natural Sobolev space is $(u,v)\in W^{1,2}([0,\pi]\times L^2([0,\pi])$. The boundary conditions $v(0)=v(\pi)=0$ then occur as natural boundary conditions. The boundary conditions for $u_x$ follow from the stationary equations. If $\epsilon=0$ this implies that $u_x(0)=u_x(\pi) =0$.
In the formulation of the $L^2$-heat flow of $\A^\epsilon_{\text{CR}}$, in the spirit of Floer's treatment of the Cauchy-Riemann equations, we impose $W^{1,2}$-regularity on both $u$ and $v$ and the zero Dirichlet boundary conditions on $v$. This choice is important for choosing the Sobolev space on which to study the linearized operator and its Fredholm theory.

\subsection{The relative Morse index}

As pointed out in the  introduction, strongly indefinite problems have infinite Morse (co)-index.
This complication defies a standard counting definition for an index. 
Instead we use the approach proposed by Floer in his treatment of the Hamilton action (e.g. see \cite{MR1703347,MR987770}). Morse index information is information about the spectrum of the linearized problem. The idea is that one can compare two linearized operators by connecting them by a continuous path of operators. This defines a Fredholm operator whose index measures a difference between `relative indices'. The concrete procedure is described below and is based on the principle of spectral flow of self-adjoint operators, cf.\ \cite{MR1331677}. 

Let $z=(u,v)$ be a solution of Problem \eqref{eq:CR_BVP1e} with $\lim_{t\to\pm \infty} z(t,\cdot) = w_\pm$, where $w_\pm$ are \emph{hyperbolic} critical points of $\A^\epsilon_{\text{CR}}$.
Linearizing the equations in \eqref{eq:CR_BVP1e} yields a linear operator 

\begin{equation}\label{eq:CR_BVP12}
\left( \begin{array}{c} \xi \\ \eta \end{array} \right) 
\mapsto
\left( \begin{array}{c} 
\xi_t  - \eta_x - \psi'_\lambda(u)\xi -\epsilon h_{uu}(x,u,v)\xi -\epsilon h_{uv}(x,u,v)\eta \\
\eta_t  +\xi_x - \eta -\epsilon h_{uv}(x,u,v)\xi -\epsilon h_{vv}(x,u,v)\eta \end{array} \right)  .
\end{equation}
Such linearized Cauchy-Riemann equations are written compactly as
\begin{equation}\label{e:LK}
L_K \bydef \partial_t-J\partial_x - K(t,x),
\end{equation}
where $J= \begin{psmallmatrix} 0 & 1 \\ -1 & 0\end{psmallmatrix}$ is the standard symplectic $2\times 2$-matrix and $K$ is a ($2\times 2$) matrix-valued function with asymptotic limits $\lim_{t\to\pm\infty} K(t,x) = K_{\pm}(x)$.
The operators $L_K$ of this type are Fredholm operators mapping $W^{1,2}(\R\times [0,\pi]) \times W^{1,2}_0(\R\times [0,\pi])$ to $L^{2}(\R\times [0,\pi])^2$, and the Fredholm index ${\rm ind}(L_K)$ only depends
on the limits $K_\pm$, which we denote by
\[
  {\rm ind} (L_K)= \iota(K_-,K_+) ,
\]
cf.\ \cite{MR1331677}.
When $J\partial_x + K_\pm = -d^2\A^\epsilon_{\text{CR}}(w_\pm)$ we define the \emph{relative Morse index} $i(w_-,w_+)$  of $w_-$ and~$w_+$ as
the Fredholm index 
\[ 
  i(w_-,w_+) \bydef \iota(K_-,K_+) ,
\]
where $K_\pm = - J\partial_x -d^2\A^\epsilon_{\text{CR}}(w_\pm)$.
The Fredholm index satisfies the co-cycle property, which expresses that concatenation of paths corresponds to addition of Fredholm indices. In particular, if $w$, $w'$ and $w''$ are critical points of $\A^\epsilon_{\text{CR}}$ then
\[
i(w,w') + i(w',w'') = i(w,w'').
\]
This property implies that the relative index function on the critical points  is well-defined. 
One may normalize
the index, for example by setting
\[
\mu(w) \bydef \iota(K, K_0),
\]
where $J \partial_x + K = -d^2\A^\epsilon_{\text{CR}}(w)$ and $K_0 = \begin{psmallmatrix} -1 & 0 \\ 0 & 1\end{psmallmatrix}$.
The operator $J \partial_x + K_0$ is hyperbolic with spectrum $\{-1\}\cup\{\pm\sqrt{n^2+1} \}_{n=1}^\infty$. With the normalized Morse index we obtain
\[
\iota(w,w') = \mu(w) - \mu(w').
\]
In the Fredholm theory for operators $L_K$ of the form~\eqref{e:LK}, the Fredholm index can be related to another characteristic of self-adjoint operators: spectral flow.
Let $\sigma \mapsto B(\sigma)$, $\sigma\in [-1,1]$, be a smooth path of self-adjoint operators such that
\[
B(\pm 1) = -d^2\A^\epsilon_{\text{CR}}(w_\pm).
\]
A path can be deformed slightly to be a \emph{generic} path, that is $B(\sigma)$ is singular only  for finitely many values of $\sigma$, \cite[Section 4]{MR1331677}.
We denote $I =\{ \sigma \in (-1,1) : B(\sigma) \text{ is singular} \}$,
where we assume that the end points $B(- 1)$ and $B(+1)$ are regular ($w_\pm$ are hyperbolic critical points).
Moreover, at any $\sigma_0 \in I$ the kernel of $B(\sigma_0)$ 
is 1-dimensional,
and the single eigenvalue such that $\lambda(\sigma_0)=0$
crosses zero transversally: $\lambda'(\sigma_0)\not =0$.
One defines the \emph{spectral flow} of a generic path as
\[
{\rm specflow}(B(\sigma))  \bydef \sum_{\sigma_0 \in I} {\rm sign}(\lambda'(\sigma_0)).
\]
The spectral flow does not depend on the chosen (generic) path, hence the definition of the spectral flow can be extended to nongeneric paths.
The spectral flow is related to the Fredholm operators $L_K$ as follows: 
\[
i(w_-,w_+) = {\rm ind}\bigl( L_K \bigr) = {\rm specflow}( J \partial_x  + K(\sigma,x)).
\]
The link between the relative index and the Fredholm operator 
is used again in the next section to determine the dimension of sets of bounded solutions.

\subsection{Isolating neighborhoods} \label{geninv12}

Problem \eqref{eq:CR_BVP1} is ill-posed when viewed as an initial value problem. 
This is typical for problems with an indefinite structure. At closer inspection the ill-posed initial value problem defines an elliptic problem on a `cylinder-like' domain. The philosophy of Morse-Floer homology is to study complete flow lines as opposed to evolving initial points.
In particular, it make sense  to consider the \emph{set of bounded solutions}
\[
\mathcal{W}_{\epsilon,h} \bydef \Bigl\{ z=(u,v) : \R\times [0,\pi] \to \R^2 \Bigm|  z\text{ solves~\eqref{eq:CR_BVP1e} and }\int_{\R\times [0,\pi]} |z_t|^2 <\infty\Bigr\}.
\]
When $\psi(u)=\Psi'(u)$ is a coercive nonlinearity, that is $\psi(u)u\le c_0-c_0|u|^{2+c_1}$, for some $c_0,c_1>0$, then we have the following compactness result:
\begin{prop} 
\label{compact1}
Suppose $\psi$ is coercive.
Then the set $\mathcal{W}_{\epsilon,h}$ is compact in $C^1_{\rm loc}(\R\times[0,\pi])$, for all $\epsilon$ and all $h$.
\end{prop}
For the stationary problem the coercivity condition readily provides a priori estimates since $\int_{[0,\pi]} v^2\le
\int_{[0,\pi]} \psi(u) u \le C$, which shows that the stationary solutions $(u,v)$ are a priori bounded in $L^\infty$.
 For the time dependent system, the estimates are slightly more involved but also lead to a priori $L^\infty$-estimates for flow lines. 
 In \cite{MR1703347} an analogue of the coercivity condition is used to obtain a priori $L^\infty$-bounds
 on flow lines in the setting of an indefinite elliptic system. The arguments for the Cauchy-Riemann equations in this paper are similar and are left to the interested reader.
 
 The compactness replaces the Palais-Smale condition for traditional variational problems and allows a finite count of critical points and certain flow-lines.
The compactness result is based on elliptic estimates and the ``geometric type'' of the nonlinearity $\psi$ (coercivity).
The latter provides an a priori $L^\infty$-bound on complete trajectories $z(t,x)$.
The global compactness result is a crucial pillar for defining invariants, cf.\ \cite{MR1703347,GVVW,MR987770}. 
Nonetheless, when such a property is not available, for example when  $\psi (u)u>0$ for $|u|\to\infty$, one way to circumvent this problem is to consider bounded solutions restricted to a subset $\mathcal{N}\subset C^1([0,\pi])$.
In fact, we will use isolating neighborhoods $\mathcal{N}$ (see Definition~\ref{def:isolating} below) even for coercive nonlinearities. 
We define
\[ 
\mathcal{W}_{\epsilon,h}(\mathcal{N}) \bydef \bigl\{ z\in \mathcal{W}_{\epsilon,h} \bigm|  z(t,\cdot)\in \mathcal{N} \text{ for all } t\in \R\bigr\}.
\]
Bounded solutions define the equivalent of an invariant set since $t$-translation of bounded solutions of the nonlinear Cauchy-Riemann equations in $\mathcal{W}_{\epsilon,h}(\mathcal{N})$
induces a (continuous) $\R$-flow on the (metric) space $\mathcal{W}_{\epsilon,h}$ (compact-open topology). We define 
\begin{equation}
\label{invset}
\mathcal{S}_{\epsilon,h}(\mathcal{N}) \bydef \bigl\{z(0,\cdot) \bigm| z\in \mathcal{W}_{\epsilon,h}(\mathcal{N}) \bigr\}\subset \mathcal{N},
\end{equation}
which is called the \emph{maximal invariant set} in $\mathcal{N}$.
Points in $\mathcal{S}_{\epsilon,h}(\mathcal{N})$ will be denoted by $w$.
The induced $\R$-flow on $\mathcal{S}_{\epsilon,h}(\mathcal{N})$ is denoted by $\phi_{\epsilon,h}$.
In the case $\epsilon=0$ we write $\mathcal{W}(\mathcal{N})$,  $\mathcal{S}(\mathcal{N})$ and $\phi\colon \R\times\mathcal{S}(\mathcal{N}) \to \mathcal{S}(\mathcal{N})$ 
for the set of bounded solutions and the induced $\R$-flow, respectively.

The Cauchy-Riemann 
equations are special in the sense that the unique continuation property yields an $\R$-flow on $\mathcal{S}(\mathcal{N})$, cf.\ \cite{MR92067,SalZeh1}.
In general, while the $t$-translation flow on $\mathcal{W}(\mathcal{N})$ always defines a $\R$-flow, for parabolic equations such as~\eqref{e:OK} the induced flow on $\mathcal{S}(\mathcal{N})$ may yield 
only a semi-flow.
The most important reason for using the induced flow $\phi$ is to have a straightforward definition of isolation and isolating neighborhood,
leading 
to a compact metric space $\mathcal{S}(\mathcal{N})$, on which the theory of attractors and Morse representations from \cite{KMV1,KMV3} can be used.

\begin{prop}[\cite{GVVW,MR987770}]
\label{compact2}
Let $\mathcal{N}\subset C^1([0,\pi])$ be a closed set. If $\mathcal{S}_{\epsilon,h}(\mathcal{N})\subset \mathcal{N}$ is bounded, then
the set $\mathcal{W}_{\epsilon,h}(\mathcal{N})$ is compact in $C^1_{\rm loc}(\R\times[0,\pi])$.
\end{prop}

As a consequence $\mathcal{S}_{\epsilon,h}(\mathcal{N})$ is a compact subset in $C^1([0,\pi])$. 
\begin{defn}[\cite{MR1703347,GVVW,MR987770}]\label{def:isolating}
A subset $\mathcal{N} \subset C^1([0,\pi])$ is called an \emph{isolating neighborhood} for Problem  \eqref{eq:CR_BVP1e},
if 
\begin{enumerate}
\item[(i)] $\mathcal{S}_{\epsilon,h}(\overline{\mathcal{N}})$ is compact;
\item[(ii)] $\mathcal{S}_{\epsilon,h}(\overline{\mathcal{N}})\subset {\rm int} (\mathcal{N})$.
\end{enumerate}
\end{defn}
A sufficient condition to guarantee
the boundedness of $\mathcal{S}_{\epsilon,h}(\mathcal{N})$ is an action bound:
\[
a\le \A^\epsilon_{\text{CR}}(z(t,x)) \le b,\quad \forall z\in \mathcal{N}, \quad a<b\in \R,
\]
cf.\ \cite{MR987770,MR1045282}.
A second important pillar for defining intrinsic invariants is a generic structure theorem for gradient systems.
We say that Problem   \eqref{eq:CR_BVP1e}  is \emph{generic} 
if (i): the critical points $w$ of $\A^\epsilon_{\text{CR}}$ are non-degenerate, i.e. $d^2\A^\epsilon_{\text{CR}}(w)$ is an invertible operator, and (ii): the \emph{adjoint} of the linearized problem \eqref{eq:CR_BVP12}
is onto 
for every bounded trajectory $z\in \mathcal{W}_{\epsilon,h}(\mathcal{N})$. 
The pair $(\epsilon,h)$ is called \emph{generic} in this case.

\begin{prop} 
\label{generic12}
For every $\epsilon\not =0$ and for almost every (in a well-defined sense) perturbation $h$, Problem  \eqref{eq:CR_BVP1e} is generic.
\end{prop}

In the case of the Cauchy-Riemann equations with periodic boundary conditions a proof of compactness is given in \cite[Section 4]{GVVW}. 
A precise statement of genericity is given in \cite[Section 9]{GVVW}. With minor adjustments one can carry out the same arguments for the Cauchy-Riemann equations with the boundary conditions used in this paper, which proves Proposition \ref{generic12}. We leave the details to the interested reader, see also \cite{SalZeh1,SalZeh2}.
In \cite[Section 9]{GVVW} the structure theorem is stated  for periodic boundary conditions and follows from the compactness and genericity. By the same token Theorem~\ref{struc12} is proved, cf.\ \cite{SalZeh1,SalZeh2}. 

When $\epsilon=0$ and $\mathcal{N}$ is an isolating neighborhood, then $\mathcal{N}$ is also isolating for $\epsilon\not=0$ sufficiently small. 
For isolating neighborhoods and generic pairs $(\epsilon,h)$ we have the following structure theorem:

\begin{thm}
\label{struc12}
Let $\mathcal{N}$ be an isolating neighborhood and let $(\epsilon,h)$ be a generic pair.
Then,
\[
\mathcal{W}_{\epsilon,h}(\mathcal{N}) = \bigcup_{w_-,w_+}\mathcal{W}_{\epsilon,h}(w_-,w_+;\mathcal{N}),
\]
where 
\[ \mathcal{W}_{\epsilon,h}(w_-,w_+;\mathcal{N}) \bydef \bigl\{z\in \mathcal{W}_{\epsilon,h}(\mathcal{N}) \bigm| \lim_{t\to\pm\infty}z(t,\cdot)=w_\pm\bigr\} ,
\]
and $w_\pm$ are (the finitely many) critical points of $\A^\epsilon_{\text{CR}}$ in $\mathcal{N}$.
The sets $\mathcal{W}_{\epsilon,h}(w_-,w_+;\mathcal{N})$ are smooth embedding manifolds (without boundary) and
\[
\dim \mathcal{W}_{\epsilon,h}(w_-,w_+;\mathcal{N}) = 
\iota(w_-,w_+).
\]
\end{thm}

The structure theorem is used in Section \ref{construction1} for the purpose of building a chain complex that defines Floer homology. Essentially the structure theorem states that the `space' of flow lines with relative index difference equal to 1 is a finite set, and the `space' of flow lines with relative index difference equal to 2 is either a closed interval or a circle. Each flow line is regarded as a `point' in the aforementioned `space'. A detailed description can be found in \cite[Section 10]{GVVW}, see also \cite{MR1239174}.

The fact that $\mathcal{N}$ is an isolating neighborhood is crucial for the above structure theorem. Without isolation the decomposition may not hold. The latter is the key ingredient for building a homology theory, cf.\ Sect. \ref{construction1}. The definition of Floer homology associated to isolating neighborhoods, not just the entire function space, dates back to Floer, cf.\ \cite{MR987770}. In \cite{GVVW} isolating neighborhoods for Floer homology are used to study Morse representations. The definition of Floer homology of a single critical point uses a neighborhoods as isolating neighborhood. Finally, sub/super level sets of the action provide isolating neighborhoods which can be used to formulate Morse relations.

\subsection{The homology construction}
\label{construction1}
Theorem \ref{struc12} states that, generically, bounded solutions are connecting orbits or critical points.
This allows us to carry out a standard construction of chain complexes.
To reduce notational clutter we fix a base point for the index and consider the normalized index $\mu(w)$.
Given an isolating neighborhood~$\mathcal{N}$ and a generic pair $(\epsilon,h)$  we define
\[
C_k(\epsilon,h;\mathcal{N}) \bydef 
 \bigoplus_{\substack{d\A^\epsilon_{\text{CR}}(w)=0 \\ \mu(w)  =k}}
 \Z_2 \langle w\rangle   ,
\]
called the $k$-dimensional \emph{chain groups} over $\Z_2$. Here we have denoted by $\langle w\rangle$ the group generator corresponding to the critical point $w$.
The latter are finite dimensional since $\mathcal{W}_{\epsilon,h}(\mathcal{N})$ is compact.
Also by compactness $\mathcal{W}_{\epsilon,h}(w_-,w_+;\mathcal{N}) $ is a finite set of trajectories whenever $i(w_-,w_+)=\mu(w_-)-\mu(w_+) =1$.
This allows us to define the boundary operator
\[
\partial_k(\epsilon,h;\mathcal{N})\colon C_k(\epsilon,h;\mathcal{N})\to C_{k-1}(\epsilon,h;\mathcal{N}),
\]
given by
\[
\partial_k\langle w\rangle \bydef \sum_{\mu(w') =k-1}n(w,w') \langle w'\rangle,
\]
where $n(w,w')\in \Z_2$ is the number of trajectories  in $\mathcal{W}_{\epsilon,h}(w_-,w_+;\mathcal{N})$
modulo $2$.
In order to justify the terminology \emph{boundary operator} we observe that 
\begin{equation}
\label{sum12}
\bigl(\partial_{k-1}\circ \partial_k\bigr) \langle w\rangle = \sum_{\mu(w'')=k-2}\sum_{\mu(w')=k-1} n(w,w')n(w',w'')\langle w''\rangle.
\end{equation}
\begin{figure}[t]
\centerline{
\includegraphics[width=1.0\textwidth]{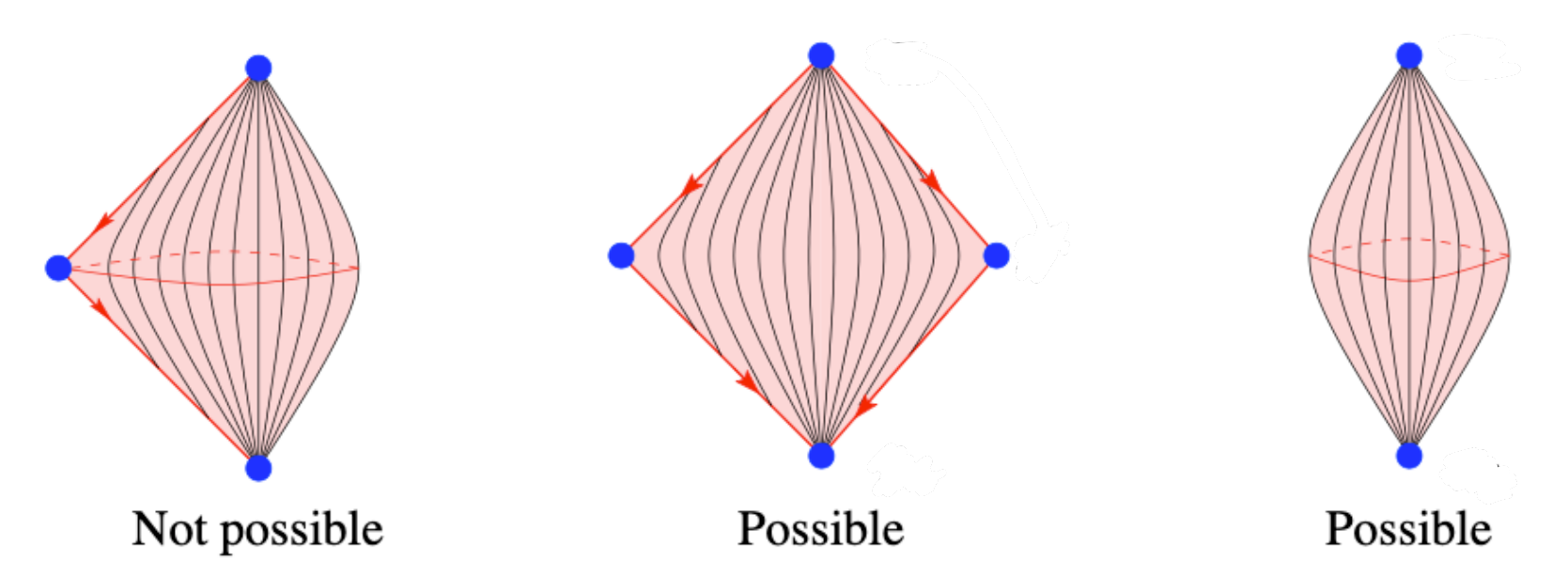}
}
\caption{The structure theorem indicates that the spaces of flow lines  with relative index difference equal to 2 can be visualized as a `diamond', or `pinched cylinder' (middle and right). Modulo time translations this yields either an closed interval of flow lines or a circle of flow lines. The scenario depicted on the left is prohibited by the structure theorem.}
\label{f:str-pos}
\end{figure}
The inner sum counts the number of 2-chain connections between $w$ and $w''$.
The structure theorem can be appended with the statement that every
(of the finitely many) components of $\mathcal{W}_{\epsilon,h}(w_-,w_+;\mathcal{N}) $ with
$i(w_-,w_+)=2$, is either a circle of trajectories or  an open interval of trajectories with distinct ends, cf.\ Fig. \ref{f:str-pos}. The latter implies that the sum in~\eqref{sum12} is even and therefore $\partial_{k-1}\circ \partial_k=0$ for all $k$, proving that
$\partial_k$ is indeed a boundary operator. Hence 
\[
\Bigl( C_k(\epsilon,h;\mathcal{N}),\partial_k(\epsilon,h;\mathcal{N})\Bigr),\quad k\in \Z,
\]
is a finite dimensional chain complex over the critical points of $\A^\epsilon_{\text{CR}}$ in $\mathcal{N}$.
The homology of the chain complex is defined as 
\begin{equation}
\label{FH12}
\HF_k(\epsilon,h;\mathcal{N}) \bydef\frac{ \ker \partial_k(\epsilon,h;\mathcal{N})}{{\rm im~} \partial_{k+1}(\epsilon,h;\mathcal{N})},
\end{equation}
which is the \emph{Floer homology} of the triple $(\epsilon,h;\mathcal{N})$.
The above construction needs the genericity since the structure Theorem~\ref{struc12} guarantees the topology of the flow line spaces discussed above. The isolation property guarantees that every flow line space associated with a closed interval contains both ends.
A priori the Floer homology depends on the three parameters $\epsilon$, $h$ and $\mathcal{N}$.
Basic properties of the Cauchy-Riemann equations can be used to show various invariance
properties of the Floer homology.

\begin{prop}[\cite{MR987770}] 
\label{homotopy1}
Let $\mathcal{N} \subset C^1([0,\pi])$ be a closed set and 
let $(\epsilon_s,h_s)_{s\in [0,1]}$ be a homotopy. Suppose that 
\begin{enumerate}
\item[(i)] $\mathcal{N}$ is an isolating neighborhood
for every pair $(\epsilon_s,h_s)$, $s\in [0,1]$;
\item [(ii)] the pairs $(\epsilon_0,h_0)$ and $(\epsilon_1,h_1)$ are generic.
\end{enumerate}
Then, $\HF_k(\epsilon_0,h_0;\mathcal{N}) \cong \HF_k(\epsilon_1,h_1;\mathcal{N})$ for all $k$, and concatenations of homotopies yield compositions of isomorphisms.
\end{prop}

We may thus interpret the Floer homology as a Conley-Floer index $\HF_*(\mathcal{N})$ of the isolating neighborhood $\mathcal{N}$.
We note that for isolating neighborhoods $\mathcal{N}$ and $\mathcal{N}'$ with the same maximal invariant set $\mathcal{S}(\mathcal{N}) = \mathcal{S}(\mathcal{N}')$ the index is the same, which motivates the definition as an index for $\mathcal{S}$:
\[
 \HF_k(\mathcal{S}) =   \HF_k(\mathcal{N})  \cong   \HF_k(\mathcal{N}') , \qquad\text{for all } k \in \Z,
\]
which can be formalized via the usual inverse limit construction.

The next step is to see how the Conley-Floer index depends on the nonlinearity $\psi$. Consider a homotopy $\psi^s$, $s\in [0,1]$, which represents a continuous family of functions $\psi^s(u)$ of superlinear polynomial growth.
\begin{thm}[Continuation, cf.\ \cite{MR1703347,MR987770,SalZeh1}] 
\label{homotopy2}
Let $\mathcal{N} \subset C^1([0,\pi])$ be a closed set and 
let $(\psi^s)_{s\in [0,1]}$ be a homotopy. Suppose $\mathcal{N}$ is isolating for all $s\in [0,1]$.
Then 
\[ 
\HF_k(\mathcal{S}_0,\psi^0) \cong \HF_k(\mathcal{S}_1,\psi^1), \qquad\text{for all }  k\in \Z,
\]
where $\mathcal{S}_0$ and $\mathcal{S}_1$ are the isolated invariant sets in $\mathcal{N}$ with respect to $\psi^0$ and $\psi^1$ respectively.
\end{thm}

In advantageous circumstances, the continuation theorem can be used to compute the Conley-Floer index, e.g.\ by continuation to a situation where there is just a single critical point (or none).
We denote the Betti numbers by 
\[
 \beta_k \bydef {\rm rank~} \HF_k\bigl(\mathcal{S}(\mathcal{N}),\psi\bigr).   
\]
Indeed, in all problems considered in the introduction, by varying parameters in the PDE (i.e.~performing a homotopy, or continuation) we arrive at a problem where we can analyze by hand the maximal invariant set $\mathcal{S}(\mathcal{N})$ completely. In our particular examples, the invariant set after continuation is then just a single homogeneous equilibrium,
hence the Betti numbers $\beta_k$ can be determined explicitly.

Furthermore, to formulate a forcing result for connecting orbits,
we assume that the number of hyperbolic critical points of relative index $k$ is bounded below by $\zeta_k$.
If $\zeta_k > \beta_k$ for some~$k$, then there must be at least one connecting orbit. We can be a bit more precise in the context where we have computationally found a finite set of hyperbolic critical points $U = \bigcup_k U_k = \bigcup_k \{u_{k,i}\}_{i=1}^{\zeta_k}$, where $u_{k,i}$ has relative index $k$. We use the notation $n_+ = \max\{n,0\}$. 
\begin{lem}[Forcing lemma]\label{l:forcing}
The number of points in $U_k$ that is not the $\alpha$ or $\omega$ limit set of any connecting orbit is at most $\beta_k$.
The number of connecting orbits with $\alpha$ or $\omega$ limit set in $U$
is bounded below by
\begin{equation}\label{e:forced}
  \frac{1}{2} \sum_{k} (\zeta_k - \beta_k)_+.
\end{equation}
\end{lem}
\begin{proof}
We outline the proof, cf.~\cite[Theorem 10.2]{BakkervdBergvdVorst}. 
Recall that we denote $U = \bigcup_k U_k = \bigcup_k \{u_{k,i}\}_{i=1}^{\zeta_k}$, where $u_{k,i}$ is hyperbolic with relative index $k$.
After reordering, let $V_k = \bigcup_{i=1}^{\xi_k} u_{k,i} \subset U_k$ consist of those points in $U_k$ that are \emph{not the $\alpha$ or $\omega$ limit set of any connecting orbit}, where thus $\xi_k \leq \zeta_k$.
The maximal invariant set $\mathcal{S}$ can be written as the disjoint union of $V=\bigcup_k V_k$ and $\hat{S}=\mathcal{S}\setminus V$. By definition, all points in $V$ are hyperbolic and $\hat{S}$ does not contain connecting orbits limiting to points in $V$. Hence $\hat{S}$ is an isolated invariant set. 
Let $\hat{\mathcal{N}}$ be an isolating neighborhood of $\hat{S}$. For sufficiently small $\epsilon>0$ the union of small balls $N_\epsilon = \bigcup_k \bigcup_i^{\xi_k} B_\epsilon(u_{k,i})$ is an isolating neighborhood of $V$, and $N_\epsilon$ is disjoint from $\hat{\mathcal{N}}$.
The disjoint union $\mathcal{N} = N_\epsilon \cup \hat{\mathcal{N}}$ is an isolating neighborhood of $\mathcal{S}$. We conclude that the Floer homology of $\mathcal{S}$ is given by 
\[
  H_k (\mathcal{S}) = H_k (V).\oplus H_k (\mathcal{\hat{S}}) 
   = (\Z_2)^{\xi_k} \oplus H_k (\mathcal{\hat{S}}).
\]
For the Betti numbers we infer that 
$\beta_k \geq \xi_k$, proving the first assertion.
Furthermore, by definition all points in $U_k \setminus V_k$ have a connecting orbit ``attached'' to it, and there are $\zeta_k - \xi_k \geq (\zeta_k -\beta_k)_+$ 
points in $U_k \setminus V_k$. 
Noticing that no more than two of the (hyperbolic) points in $U$ can be in the union of the
$\alpha$ and $\omega$ limit set of a single connecting orbit, we
arrive at the lower bound~\eqref{e:forced}.
\end{proof}

While for the cases encountered in our applications we are satisfied with the forcing result provided by this lemma, there is definitely room for improvement. For example, the ordering in terms of energy of the critical points contains additional information that can lead to stronger forcing results. In such situations one will need a refinement of the setup in terms of Morse representations, which we may indeed introduce in the current (Morse-Conley-Floer) context along the lines presented in~\cite{KMV1,KMV3}. 
We leave this for future work.


\section{Computing the equilibria and their relative indices} \label{s:setuphomotopy}

In this section, we introduce, in the context of the Cauchy-Riemann action
functional $\A_{\text{CR}}$, the rigorous computational method to obtain the
critical points and their relative indices. The approach is analogous for the
action functionals $\A_{\text{TW}}$ and $\A_{\text{OK}}$, see
Sections~\ref{s:TW} and~\ref{s:OK} for a discussion of the minor differences.
In Section~\ref{sec:rig_comp_critical points}, we introduce the rigorous method
to compute the critical points. This is done by solving a problem of the form
$F(a)=0$ posed on a Banach space of Fourier coefficients $a=(a_k)_k$ decaying
to zero geometrically. Then, in Section~\ref{sec:rig_comp_spectrum} we
introduce a method to control the spectrum of the derivative $DF(\ta)$ for each
rigorously computed critical points $\ta$. Finally, in
Section~\ref{sec:rig_comp_relative_indices}, we show how to compute rigorously
the relative index of critical points.

\subsection{Computation of the critical points}  \label{sec:rig_comp_critical points}

Studying a critical point $(u(x),v(x))$ of the action functional $\A_{\text{CR}}$ reduces to study the steady states (time independent solutions) of Problem \eqref{e:CR}, that is
\begin{equation}\label{eq:CR_BVP}
\begin{cases}
0 = v_x + \lambda_1 u- \lambda_2 u^3 ,  \\
0 = -u_x + v  , \\
u_x(0)=u_x(\pi)=0,\\
v(0)=v(\pi)=0.
\end{cases}
\end{equation}
Here we have added the redundant Neumann boundary conditions for $u$ (they follow immediately from the second equation) to make the symmetries more obvious. To be explicit, any solution pair $(u,v)(x)$ of~\eqref{eq:CR_BVP} can be extended to a periodic solution of the ODE system for $x\in \R$ with the symmetries
\[
u(-x)=u(x) \qquad\text{and}\qquad v(-x)=-v(x).
\]
Due to these symmetries, any solution $(u,v)$ of \eqref{eq:CR_BVP} can be expressed using the Fourier expansions
\begin{subequations}
\label{e:pairFourier}
\begin{align}
u(x) &= \sum_{k \in \Z} (a_1)_k e^{i k x}, && \quad (a_1)_k \in \R ~~ \text{and} ~~ (a_1)_{-k} =  (a_1)_k  ~\text{ for } k>0,
\\
v(x) &= \sum_{k \in \Z} i (a_2)_k e^{i k x}, && \quad (a_2)_k \in \R ~~ \text{and} ~~ (a_2)_{-k} =  -(a_2)_k  ~\text{ for } k>0,
\end{align}
\end{subequations}
{which are essentially a cosine series for $u(x)$ and a sine series for $v(x)$.}

There are several ways to transform the problem~\eqref{eq:CR_BVP} to the Fourier setting. Since the variational formulation is a crucial viewpoint, we choose to start by writing the action in terms of the Fourier (or rather cosine and sine) coefficients explicitly:
\[ 
  \A_{\text{CR}} (a_1,a_2) =  2 \sum_{k=1}^\infty k (a_1)_k (a_2)_k - \sum_{k=1}^\infty (a_2)_k^2 -  \frac{\lambda_1}{2} (a_1^2)_0 + \frac{\lambda_2}{4} (a_1^4)_0 ,
\]
where $a_1=\{(a_1)_k\}_{k \geq 0}$ and $a_2=\{(a_2)_k\}_{k > 0}$ are real variables. This creates a notationally inconvenient asymmetry between $a_1$ and $a_2$, and we use $(a_2)_0=0$ throughout without further ado.
The convolution powers are the natural ones stemming from the convolution product
\begin{equation}\label{e:convolution}
  (a_1 \ast \tilde{a}_1)_k = \sum_{k' \in \Z} (a_1)_{k'} (\tilde{a}_1)_{k-k'},
\end{equation}
when taking into account the  symmetries in~\eqref{e:pairFourier}, for example
\begin{alignat*}{1}
(a_1^2)_0 & = (a_1)_0^2 + 2 \sum_{k=1}^\infty (a_1)_k^2, \\ 
(a_1^4)_0 & = \sum_{k_1+k_2+k_3+k_4=0 \atop k_i \in \mathbb Z} (a_1)_{|k_1|} (a_1)_{|k_2|} (a_1)_{|k_3|} (a_1)_{|k_4|}.
\end{alignat*}
We have scaled out an irrelevant factor $\pi$ in the action compared to~\eqref{eq:functional_CR}. 
We choose as the inner product
\begin{equation}\label{e:innerproduct}
  \bigl\langle (a_1,a_2), (\tilde{a}_1,\tilde{a}_2) \bigr\rangle 
  \bydef \sum_{k \geq 0} (a_1)_k (\tilde{a}_1)_k + \sum_{k > 0} (a_2)_k (\tilde{a}_2)_k,
\end{equation}
so that the Hessian will have the straightforward appearance of a symmetric matrix (when restricted to natural finite dimensional projections).

\begin{rem}
We note that the alternative inner product
\[
  \bigl\langle\!\bigl\langle (a_1,a_2), (\tilde{a}_1,\tilde{a}_2) \bigr\rangle\!\bigr\rangle 
  \bydef 
  (a_1)_0 (\tilde{a}_1)_0 + 2 \sum_{k >0 } (a_1)_k (\tilde{a}_1)_k + 2 \sum_{k > 0} (a_2)_k (\tilde{a}_2)_k,
\]
is the one corresponding to the $L^2$ inner product in function space, which was used to interpret the Cauchy-Rieman equations~\eqref{e:CR} as the negative gradient flow of the action functional~\eqref{eq:functional_CR}. 
In terms of reading of symmetry properties from matrix representations this alternative inner product is less convenient, although this could be remedied by rescaling $(a_i)_k$ for $k>0$ by a factor $\sqrt{2}$.
On the other hand, a disadvantage of using such rescaled Fourier coefficients is that it would complicate the description of the convolution product.
Since the relative index is independent of the particular choice of inner product, we choose to work with~\eqref{e:innerproduct} in the setup for the relative index computations.
\end{rem}

We write $a=(a_1,a_2)$. Taking the negative gradient of $\A_{\text{CR}}$ with respect to the inner product~\eqref{e:innerproduct}, we arrive at the system  
\begin{equation}\label{eq:F=0_CR}
\left\{ 
\begin{array}{rlll}
(F_1(a))_0 &\!\!\!\!\bydef \lambda_1 (a_1)_0 - \lambda_2 (a_1^3)_0 &=0 \\[1mm]
(F_1(a))_k &\!\!\!\!\bydef
2[-k (a_2)_k + \lambda_1 (a_1)_k - \lambda_2 (a_1^3)_k ] 
&=0& \qquad \text{for } k>0,\\[1mm]
(F_2(a))_k &\!\!\!\!\bydef
2[-k (a_1)_k + (a_2)_k]
&=0& \qquad \text{for } k > 0.
\end{array}\right. 
\end{equation}
We use~\eqref{eq:F=0_CR} as the Fourier equivalent of~\eqref{eq:CR_BVP}.
The factors 2 in~\eqref{eq:F=0_CR} for $k>0$ are the result of the symmetries in~\eqref{e:pairFourier} in combination with the inner product choice~\eqref{e:innerproduct}.   

We set $F(a) = (\{F_1(a)\}_{k \ge 0},\{F_2(a)\}_{k > 0})$. 
Consider the Banach spaces (i.e., unrelated to the inner product)
\begin{equation} \label{eq:ell_nu_one}
\ell^1 \bydef \left\{ \tilde{a} = (\tilde{a}_k)_{k \ge 0} : \|\tilde{a}\|_{1} \bydef |\tilde{a}_0| + 2 \sum_{k=1}^\infty |\tilde{a}_k|   < \infty \right\},
\end{equation}
and $\ell^{1,0} = \{ \tilde{a}\in \ell^1 : \tilde{a}_0=0 \}$,
and define $X \bydef \ell^1 \times \ell^{1,0}$,
with the induced norm, given $a=(a_1,a_2) \in X$,
\begin{equation} \label{eq:normX_CR}
\|a\|_X \bydef \max\{ \|a_1\|_{1},\|a_2\|_{1} \}.
\end{equation}
The problem of looking for solutions of \eqref{eq:CR_BVP} therefore reduces to finding $a \in X$ such that $F(a)=0$, where the map $F$ is defined component-wise in \eqref{eq:F=0_CR}. Solving the problem $F=0$ in $X$ is done using computer-assisted proofs. The following Newton-Kantorovich type theorem provides an efficient means of performing that task. 

Denote by $B_r(a) \bydef \{ x \in X : \| x - a \|_X \le r\}$ the closed ball of radius $r>0$ centered at a given $a \in X$. 

\begin{thm}[{\bf A Newton-Kantorovich type theorem}] \label{thm:radii_polynomials}
Let $X$ and $X'$ be Banach spaces, $A^{\dagger} \in B(X,X')$ and $A \in B(X',X)$ be bounded linear operators.
Assume $F \colon X \to X'$ is Fr\'echet differentiable at $\ba \in X$, $A$ is injective and $A F \colon X \to X.$
Let $Y_0$, $Z_0$ and $Z_1$ be nonnegative constants, and let $Z_2:(0,\infty) \to (0,\infty)$ be a function satisfying
\begin{align}
\label{eq:general_Y_0}
\| A F(\ba) \|_X &\le Y_0
\\
\label{eq:general_Z_0}
\| I - A A^{\dagger}\|_{B(X)} &\le Z_0
\\
\label{eq:general_Z_1}
\| A[A^{\dagger} - DF(\ba)] \|_{B(X)} &\le Z_1,
\\
\label{eq:general_Z_2}
\| A[DF(c) - DF(\ba)]\|_{B(X)} &\le Z_2(r) r, \quad \text{for all } c \in B_r(\ba),
\end{align}
where $\| \cdot \|_{B(X)}$ denotes the operator norm.  Define the radii polynomial by 
\begin{equation} \label{eq:general_radii_polynomial}
p(r) \bydef Z_2(r) r^2 - ( 1 - Z_1 - Z_0) r + Y_0.
\end{equation}
If there exists $r_0>0$ such that $p(r_0)<0$, then there exists a unique $\ta \in B_{r_0}(\ba)$ such that $F(\ta) = 0$.
\end{thm}

\begin{proof}
The idea of the proof (for the details, see Appendix A in \cite{MR3612178}) is to show that the Newton-like operator $T(a) \bydef a-AF(a)$ satisfies $T:B_{r_0}(\ba) \to B_{r_0}(\ba)$ and it is a contraction mapping, that is, there exists $\kappa \in [0,1)$ such that $\|T(x)-T(y)\|_X \le \kappa \| x - y \|_X$, for all $x,y \in B_{r_0}(\ba)$. The proof follows from Banach fixed point theorem. 
\end{proof}

Proving the existence of a solution of $F=0$ using Theorem~\ref{thm:radii_polynomials} is often called the {\em radii polynomial approach} (see e.g.~\cite{MR2338393,MR2443030}). In practice, this approach consists of considering a finite dimensional projection of \eqref{eq:F=0_CR}, computing an approximate solution $\ba$ (i.e. such that $F(\ba) \approx 0$), considering an approximation $A^\dag$ of the derivative $DF(\ba)$ and an approximate inverse $A$ of $DF(\ba)$. Once the numerical approximation $\ba$ and the linear operators $A$ and $A^\dag$ are obtained, formulas for the bounds $Y_0$, $Z_0$, $Z_1$ and $Z_2(r)$ are derived analytically and finally implemented in a computer-program using interval arithmetic (see \cite{MR0231516}). The final step is to find (if possible) a radius $r_0>0$ for which $p(r_0)<0$. In case such an $r_0$ exists, it naturally provides a $C^0$ bound for the error between the approximate solution $(\bar u(x),\bar v(x))$ and the exact solution $(\tilde u(x),\tilde v(x))$, which have Fourier coefficients $\bar{a}$ and $\tilde{a}$, respectively, see~\eqref{e:pairFourier}.
The following remark makes this statement explicit. 

\begin{rem} [{\bf Explicit error control}] \label{rmk:error_control}
Assume that $r_0>0$ satisfies $p(r_0)<0$, where $p$ is the radii polynomial defined in \eqref{eq:general_radii_polynomial}.
Then the unique $\ta \in B_{r_0}(\ba)$ such that $F(\ta) = 0$ satisfies
\[
\| \ta - \ba \|_X = \max \left\{  \| \ta_1 - \ba_1 \|_{1}, \| \ta_2 - \ba_2 \|_{1} \right\} \le r_0,
\] 
which implies that
\begin{align*}
\| \tilde u - \bar u \|_{C^0} &= \sup_{x \in [0,\pi]} | \tilde u(x) - \bar u(x) | 
= \sup_{x \in [0,\pi]} \left| \sum_{k \in \Z} [ (\tilde a_1)_k - (\bar a_1)_k ] e^{i k x} \right| 
\\ & 
\le \sum_{k \in \Z} | (\tilde a_1)_k - (\bar a_1)_k | 
= \| \ta_1 - \ba_1 \|_{1} 
\le r_0.
\end{align*}
Analogously, we obtain the bound $\| \tilde v - \bar v \|_{C^0} \le r_0$. Higher regularity of $\tilde{u}$ and $\tilde{v}$ and corresponding error bounds on derivatives can be established through classical elliptic regularity arguments. Alternatively, smoothness can be built into the space of Fourier coefficients by choosing an exponentially weighted $\ell^1$ space, see e.g.~\cite{HLM}, but we aimed to keep computational details to an minimum in the current paper.
\end{rem}

As mentioned previously, the radii polynomial approach begins by computing an approximate solution $\ba$ of $F=0$. This first requires considering a finite dimensional projection. Fixing a projection size $m \in \N$, denote a finite dimensional projection of $a \in X$ by $a^{(m)} = \big( ((a_1)_k)_{k=0}^{m-1} ,((a_2)_k)_{k=1}^{m-1} \big) \in \R^{2m-1}$. The finite dimensional projection of $F$ is then given by $F^{(m)}=(F_1^{(m)},F_2^{(m)}):\R^{2m-1} \to \R^{2m-1}$ defined by 
\begin{equation} \label{eq:projection_F_CR}
F^{(m)}(a^{(m)}) \bydef 
\begin{bmatrix}
\left( F_1(a^{(m)})_k \right)_{0\leq k<m}
\\
\left( F_2(a^{(m)})_k \right)_{1\leq k<m}
\end{bmatrix}.
\end{equation}
Assume that a solution $\ba^{(m)}$ such that $F^{(m)}(\ba^{(m)}) \approx 0$ has been computed (e.g.~using Newton's method).
Given $i=1,2$, denote by $\ba_i = \left( (\ba_i)_0 ,\dots,(\ba_i)_{m-1},0,0,0,\dots \right)$ the vector which consists of embedding $\ba_i^{(m)} \in \R^m$ in the infinite dimensional space $\ell^1$ by {\em padding} the tail by infinitely many zeroes. We recall we set $(\ba_2)_0=0$ by symmetry convention. Denote $\ba = (\ba_1,\ba_2)$, and for the sake of simplicity of the presentation, we use the same notation $\ba$ to denote $\ba \in X$ and $\ba^{(m)} \in \R^{2m-1}$. Denote by $DF^{(m)}(\ba)$ the Jacobian of $F^{(m)}$ at $\ba$, and let us write it as
\[
DF^{(m)}(\ba)=
\begin{pmatrix}
D_{a_1} F_1^{(m)}(\ba) & D_{a_2} F_1^{(m)}(\ba)\\
D_{a_1} F_2^{(m)}(\ba) & D_{a_2} F_2^{(m)}(\ba)
\end{pmatrix} \in M_{2m-1}(\R).
\]
The next step is to construct the linear operator $A^\dag$ (an approximation of the derivative $DF(\ba)$), and the linear operator $A$ (an approximate inverse of $DF(\ba)$). Let
\begin{equation} \label{eq:dagA_CR}
A^\dagger=
\begin{pmatrix}
A_{1,1}^\dagger & A_{1,2}^\dagger\\
A_{2,1}^\dagger & A_{2,2}^\dagger
\end{pmatrix} ,
\end{equation}
whose action on an element $h=(h_1,h_2) \in X$ is defined by $(A^\dagger h)_i = A_{i,1}^\dagger h_1 + A_{i,2}^\dagger h_2$, for $i=1,2$. Here the action of $A_{i,j}^\dagger$ is defined as
\begin{align*}
(A_{i,1}^\dagger h_1)_k &= 
\begin{cases}
\bigl(D_{a_1} F_i^{(m)}(\ba) h_1^{(m)} \bigr)_k &\quad\text{for }   0 \leq k < m ,  \\
-\delta_{i,2} k (h_1)_k  &\quad\text{for }  k \ge m,
\end{cases}
\\
(A_{i,2}^\dagger h_2)_k &= 
\begin{cases}
\bigl(D_{a_2} F_i^{(m)}(\ba) h_2^{(m)} \bigr)_k &\quad\text{for }   1 \leq k < m,   \\
-\delta_{i,1} k (h_2)_k  &\quad\text{for }  k \ge m,
\end{cases}
\end{align*}
where $\delta_{i,j}$ is the Kronecker $\delta$. 
Consider now a matrix $A^{(m)} \in M_{2m-1}(\R)$ computed so that $A^{(m)} \approx {DF^{(m)}(\ba)}^{-1}$. We decompose it into four blocks:
\[
A^{(m)}=
\begin{pmatrix}
A_{1,1}^{(m)} & A_{1,2}^{(m)}\\
A_{2,1}^{(m)} & A_{2,2}^{(m)}
\end{pmatrix}.
\]
This allows defining the linear operator $A$ as
\begin{equation} \label{eq:A_CR}
A=
\begin{pmatrix}
A_{1,1} & A_{1,2}\\
A_{2,1} & A_{2,2}
\end{pmatrix} ,
\end{equation}
whose action on an element $h=(h_1,h_2) \in X$ is defined by $(Ah)_i = A_{i,1} h_1 + A_{i,2} h_2$, for $i=1,2$. 
The action of $A_{i,j}$  is defined as
 \begin{align*}
 (A_{i,1} h_1)_k &=
 \begin{cases}
 \left(A_{i,1}^{(m)} h_1^{(m)} \right)_k & \text{for }   0 \leq k < m    \\
- \delta_{i,2} \frac{1}{k} (h_1)_k  & \text{for }  k \ge m 
 \end{cases}
 \\
 (A_{i,2} h_2)_k &=
 \begin{cases}
 \left( A_{i,2}^{(m)}  h_2^{(m)} \right)_k & \text{for }   1 \leq k < m   \\
 -\delta_{i,1} \frac{1}{k} (h_2)_k  & \text{for }  k  \ge m.
 \end{cases}
 \end{align*}
Having obtained an approximate solution $\ba$ and the linear operators $\dagA$ and $A$, the next step is to construct the bounds $Y_0$, $Z_0$, $Z_1$ and $Z_2(r)$ satisfying \eqref{eq:general_Y_0}, \eqref{eq:general_Z_0}, \eqref{eq:general_Z_1} and \eqref{eq:general_Z_2}, respectively. Their analytic derivation will be done explicitly in Section~\ref{s:CR} for the Cauchy-Riemann equations. Assume that using these explicit bounds we applied the radii polynomial approach and obtained $r_0>0$ such that $p(r_0)<0$. As the following remark shows, this implies that $A$ is an injective operator. 

\begin{rem}[{\bf Injectivity of the linear operator \boldmath$A$\unboldmath}] \label{remark:injectivity_of_A}
If $r_0>0$ satisfies $p(r_0)<0$, then $Z_2(r_0) r_0^2 + (Z_0 + Z_1) r_0 + Y_0 < r_0$. Since $Y_0$, $Z_0$, $Z_1$ and $Z_2(r_0)$ are nonnegative, this implies that $\| I - A A^{\dagger}\|_{B(X)} \le Z_0 <1$. By construction of the linear operators $A$ and $\dagA$, this implies that
\[
\| I_{\R^{2m-1}} - A^{(m)} DF^{(m)}(\ba) \| < 1,
\]
which in turns implies that both $A^{(m)}$ and $DF^{(m)}(\ba)$ are invertible matrices in $M_{2m-1}(\R)$. Since $A^{(m)}$ is invertible and the tail part of $A$ is invertible by construction, $A$ is injective. 
\end{rem}

As consequence of Remark~\ref{remark:injectivity_of_A}, if $r_0>0$ satisfies $p(r_0)<0$, then $A$ is injective, and therefore there exists a unique $\ta \in B_{r_0}(\ba)$ such that $F(\ta)=0$.

\begin{rem}
Using a finite dimensional projection of size $m=100$ we computed two numerical approximations $\ba^{(1)}$ and $\ba^{(2)}$. In Figure~\ref{f:CR}, the approximate solutions $\ba^{(1)}$ (left) and $\ba^{(2)}$ (right) are plotted. 
For each approximation, the code {\tt script\_proofs\_CR.m} (available at \cite{codes_webpage}) computes with interval arithmetic (using INTLAB, see \cite{Ru99a}) the bounds $Y_0$, $Z_0$, $Z_1$ and $Z_2$ using the explicit formulas presented in Section~\ref{s:CR}. For each $\ba^{(i)} $ ($i=1,2$), the code verifies that $p(r_0^i)<0$. From this, we conclude that there exists $\ta^{(i)} \in X$ such that $F(\ta^{(i)})=0$ and such that $\| \ta^{(i)} - \ba^{(i)} \|_X \le r_0^i$, where
$r_0^1 = 4.7 \cdot 10^{-11}$
and
$r_0^2 = 1.1 \cdot 10^{-13}$.
\end{rem}

Having introduced the ingredients to compute a critical point $\ta$, we now turn to the question of controlling the spectrum of $DF(\ta)$.

\subsection{Controlling the spectrum of \boldmath $DF(\ta)$ \unboldmath} \label{sec:rig_comp_spectrum}

Due to the boundary conditions $v(0)=v(\pi)=0$, all eigenfunction pairs of the linearization around an equilibrium $(\tilde{u},\tilde{v})$ of~\eqref{eq:CR_BVP1} can smoothly be extended $2\pi$-periodically. Like the equilibrium itself, the eigenfunction pairs have the even/odd symmetries corresponding to cosine and sine series. This implies that the eigenvalues of $L_K(\tilde{u},\tilde{v})$ in function space $W^{1,2} \times W^{1,2}_0$
are in one-to-one correspondence with the eigenvalues of $DF(\tilde{a})$ in Fourier space $X=\ell^1 \times \ell^1_0$. In this section we analyse the spectral flow between the linearizations at two different critical points choosing a path that is represented in Fourier space.

We assume that using the radii polynomial approach of Theorem~\ref{thm:radii_polynomials}, we have proven existence of a unique $\ta \in B_{r_0}(\ba) \subset X$ such that $F(\ta)=0$ for some $r_0>0$ satisfying $p(r_0)<0$. Denote by $DF(\ta)$ the derivative at $\ta$. Recall that when we rigorously compute this solution we use an operator $\dagA$, defined by \eqref{eq:dagA_CR}, which approximates the Jacobian~$DF(\ba)$.

Both the Hessian $DF(a)$ and the approximation $\dagA$  of $DF(\ba)$ are symmetric with respect to the inner product~\eqref{e:innerproduct}, hence their eigenvalues are real-valued. 
Given any $c \in X$, we define the homotopy between $DF(c)$ and $\dagA$ by
\begin{equation} \label{eq:homotopy_dagA_DF}
\D_{c}(\sigma) \bydef (1-\sigma) DF(c)+ \sigma A^\dagger , \qquad\text{for } \sigma \in [0,1].
\end{equation}

\begin{thm} \label{thm:homotopy_dagA_DF}
Assume that $r_0>0$ satisfies $p(r_0)<0$ with $p$ given in \eqref{eq:general_radii_polynomial}.
For any $c \in B_{r_0}(\ba)$, we have
${\rm specflow}(\D_c(\sigma))  = 0$. Moreover, the unique zero $\tilde{a}$ of $F$ in $B_{r_0}(\ba)$ is a hyperbolic critical point.
\end{thm}

\begin{proof}
From the hypothesis that $p(r_0)<0$, we obtain
\begin{equation}\label{e:sumZ}
Z_2(r_0) r_0 + Z_0 + Z_1 + \frac{Y_0}{r_0} = \frac{1}{r_0} ( Y_0 + (Z_0 + Z_1)r_0 + Z_2(r_0) r_0^2) < 1. 
\end{equation}
Hence
\begin{equation} \label{eq:bounds_less_than_one}
\| I - A A^\dagger \|_X \le Z_0 <1  \quad \text{and} \quad \sup_{c \in B_{r_0}(\ba)} \| I - A DF(c) \|_X <1,
\end{equation}
where the first inequality follows from the fact that $Z_0<1$, and the second inequality holds since, for any $c \in B_{r_0}(\ba)$,
\begin{align*}
\| I - A DF(c) \|_X &= 
\| I - A A^{\dagger} +  A[A^{\dagger} - DF(\ba)] + A[DF(\ba) - DF(c)]\|_X  \\
& \le \| I - A A^{\dagger}\|_X + \| A[A^{\dagger} - DF(\ba)] \|_X + \| A[DF(\ba) - DF(c)]\|_X \\
& \le Z_0 + Z_1 + Z_2(r_0) r_0 < 1,
\end{align*}
where the final inequality follows from~\eqref{e:sumZ}.
Hence, given any $c \in B_{r_0}(\ba)$ and any $\sigma \in [0,1]$,
\begin{align*}
\|I - A \D_c(\sigma) \|_X &= \| I - A (\sigma A^\dagger + (1-\sigma) DF(c)) \|_X
\\ &= \| (\sigma I+(1-\sigma)I - A (\sigma A^\dagger + (1-\sigma) DF(c)) \|_X  \\
&\le \sigma \| I - AA^\dagger \|_X + (1-\sigma) \|I - A DF(c)) \|_X \\
&< \sigma + (1-\sigma) = 1.
\end{align*}
By a standard Neumann series argument, the composition $A \D_c(\sigma)$ is invertible. We infer that $\D_c(\sigma)$ is injective  for all $\sigma \in [0,1]$. Hence ${\rm specflow}(\D_c(\sigma))  = 0$. Moveover, for the Hessian $DF(\tilde{a})$ at the critical point $\tilde{a} \in B_{r_0}(\ba)$ we have ${\rm ker}(DF(\tilde{a})) = {\rm ker}(\D_{\tilde{a}}(0)) = \{0\}$, implying that $\tilde{a}$ is hyperbolic.
\end{proof}

Assume that we have proven the existence of two critical points $\ta$ and $\tb$ of the Cauchy-Riemann problem \eqref{eq:F=0_CR} using the radii polynomial approach (Theorem~\ref{thm:radii_polynomials}). Denote by $\ba$ and $\bb$ the numerical approximation of $\ta$ and $\tb$,
and by $\dagA_{\ba}$ and $\dagA_{\bb}$ the approximate derivatives
used to obtain the computer-assisted proofs.
In addition to the paths $\D_{\ta}(\sigma)$ and $\D_{\tb}(\sigma)$ discussed above, we introduce the following  paths of linear operators:
\begin{alignat*}{2}
  \D_{\ta \to \tb} (\sigma) &=   
  (1-\sigma) DF(\ta)+ \sigma DF(\tb)   ,
  &\quad&\text{for } \sigma \in [0,1],\\
  \D^\dagger_{\ba \to \bb} (\sigma) &= 
   (1-\sigma) A^\dagger_{\ba}  + \sigma A^\dagger_{\bb},
   &\quad&\text{for } \sigma \in [0,1].
\end{alignat*}
To compute the relative index of $\ta$ and $\tb$ we use the identity
\begin{alignat}{1}
   i\bigl(\ta,\tb\bigr) &= {\rm specflow}\bigl(\D_{\ta \to \tb} (\sigma)\bigr) \nonumber\\
   &={\rm specflow}\bigl(\D_{\ta}(\sigma)\bigr) + {\rm specflow}\bigl(\D^\dagger_{\ba \to \bb} (\sigma)\bigr)- {\rm specflow}\bigl(\D_{\tb}(\sigma)\bigr) \nonumber \\
   &= {\rm specflow}\bigl(\D^\dagger_{\ba \to \bb} (\sigma)\bigr), \label{e:specflowdagger}
\end{alignat}
where we have used independence with respect to the chosen path, as well as Theorem~\ref{thm:homotopy_dagA_DF}.
In the next section we discuss how to compute
the spectral flow in the righthand side of~\eqref{e:specflowdagger}. 

\subsection{Computing the relative indices} \label{sec:rig_comp_relative_indices}

To continue the discussion from Section~\ref{sec:rig_comp_spectrum},
we assume that we have proven the existence of two critical points $\ta$ and $\tb$ of the Cauchy-Riemann problem \eqref{eq:F=0_CR} using the radii polynomial approach (Theorem~\ref{thm:radii_polynomials})
in balls around the numerical approximations $\ba$ and $\bb$.
We denote by $m_{\ba}$ and $m_{\bb}$ the dimensions of the finite dimensional  projections used, and we set $m=\max \{ m_{\ba} , m_{\bb} \}$.

Ordering the components of  $a$ as 
\[
a  = \left( (a_1)_0,
(a_1)_1, (a_2)_1,\dots,(a_1)_k, (a_2)_k,\dots \right)
\]
 leads to the following representation 
 of the linear operator $\dagA_{\ba}$:
\begin{equation*}
\dagA_{\ba} =
\begin{pmatrix}
DF^{(m_{\ba})}(\ba) &                        &                    & \\
                    & \Lambda_{m_{\ba}} &                     & \\
                    &                       & \Lambda_{m_{\ba}+1} & \\
                    &                       &                     & \ddots 
\end{pmatrix}, \qquad \Lambda_k \bydef \begin{pmatrix} 0 & -k \\ -k & 0 \end{pmatrix},
\end{equation*}
and similarly for $\dagA_{\bb}$.
Alternatively, we may write
\begin{equation*} 
\dagA_{\ba} =
\begin{pmatrix}
(\dagA_{\ba})^{(m)} &                        &                    & \\
                    & \Lambda_{m} &                     & \\
                    &                       & \Lambda_{m+1} & \\
                    &                       &                     & \ddots 
\end{pmatrix},
\end{equation*}
and similarly for $\dagA_{\bb}$.
The latter representation allows us to write the homotopy
\[
 \D^\dagger_{\ba \to \bb} (\sigma) 
 =
 \begin{pmatrix}
 (1-\sigma) (\dagA_{\ba})^{(m)} + \sigma (\dagA_{\bb})^{(m)}  &                        &                    & \\
                     & \Lambda_{m} &                     & \\
                     &                       & \Lambda_{m+1} & \\
                     &                       &                     & \ddots 
 \end{pmatrix}.
\]
 
The tail of $\D^\dagger_{\ba \to \bb} (\sigma)$ is independent of $\sigma$ 
and has eigenvalues $\{ \pm k : k \ge m\}$. 
Hence, any crossing of eigenvalues of 
$\D^\dagger_{\ba \to \bb} (\sigma)$ must come from the finite dimensional part
\[
  (\D^\dagger_{\ba \to \bb})^{(m)} (\sigma)  = (1-\sigma) (\dagA_{\ba})^{(m)} + \sigma (\dagA_{\bb})^{(m)} .
\]
We may perturb this finite dimensional path to a generic one
to conclude that 
\begin{alignat}{1}
	{\rm specflow}\bigl(\D^\dagger_{\ba \to \bb} (\sigma)\bigr)
	& =
	{\rm specflow}\bigl( (\D^\dagger_{\ba \to \bb})^{(m)} (\sigma) \bigr) \nonumber \\
	& =  n_{2m-1}\bigl((\dagA_{\ba})^{(m)} \bigr) -  n_{2m-1}\bigl((\dagA_{\bb})^{(m)} \bigr), 
	\label{e:counteigenvalues}
\end{alignat}
where $n_{2m-1}(Q)$ denotes the number of positive eigenvalues of a $(2m-1) \times (2m-1)$ matrix~$Q$.

Using the tools of validated numerics, 
one can use interval arithmetic and the contraction mapping theorem (e.g. via the method \cite{MR3204427}) 
to enclose rigorously all eigenvalues of $(\dagA_{\ba})^{(m)}$ and therefore compute $n_{2m-1}\bigl((\dagA_{\ba})^{(m)} \bigr)$.
A fast alternative, especially when the projection dimension is large and there are repeated eigenvalues (for example due to symmetry, such as in Problem~\eqref{e:cylinder} posed on the square $[0,\pi]^2$, see also Section~\ref{s:TW})
or when $m$ is large, is to determine $n_{2m-1}(Q)$  via a similarity argument (cf.~\cite{OKspacegroups}). Namely, one can determine a basis transformation $V$ using approximate eigenvectors of $Q$, enclose the inverse of $V$ by interval arithmetic methods, and compute the (interval-valued) matrix $Q_0 = V^{-1} Q V$. Then $Q_0$ has the same eigenvalues as~$Q$, and it is approximately diagonal. When none of the Gershgorin circles associated to~$Q_0$ intersect the imaginary axis, one may read of $n_{2m-1}(Q_0) = n_{2m-1}(Q)$ from the diagonal of~$Q_0$. We note that as another alternative (repeated) eigenvalues can also be verified conveniently with the built-in eigenvalue verification routine of INTLAB~\cite{Ru99a}.

In conclusion, by combining~\eqref{e:specflowdagger} and~\eqref{e:counteigenvalues} we find that
the (computable) formula
\begin{equation*}
i\bigl(\ta,\tb\bigr) =n_{2m-1}\bigl((\dagA_{\ba})^{(m)} \bigr) -  n_{2m-1}\bigl((\dagA_{\bb})^{(m)} \bigr)
\end{equation*}
for the relative index of~$\ta$ and $\tb$.


\section{The bounds for the Cauchy-Riemann equations} \label{s:CR}

In this section, we present the explicit construction of the bounds $Y_0$, $Z_0$, $Z_1$ and $Z_2(r)$ satisfying \eqref{eq:general_Y_0}, \eqref{eq:general_Z_0}, \eqref{eq:general_Z_1} and \eqref{eq:general_Z_2}, in the context of 
the Cauchy-Riemann zero finding problem $F=(F_1,F_2)=0$ given in \eqref{eq:F=0_CR}.
Denote by $\ba$ an approximate solution of $F=0$, and recall from \eqref{eq:dagA_CR} and \eqref{eq:A_CR} the definition of the linear operators $\dagA$ and $A$, respectively.

Before proceeding with the presentation of the bounds, we begin by introducing some elementary functional analytic results useful for the computation of the bounds. We omit the elementary proofs, which can mostly be found  in \cite{HLM}. We use the convention $(F_2)_0=0$ whenever convenient. To keep the notation light, throughout we implicitly use the projection $\pi_0 : \ell^1 \to \ell^{1,0}$ and natural embedding $\iota_0 : \ell^{1,0} \to \ell^{1}$ liberally, e.g., without further ado we identify  an operator $\Gamma_0 \in B(\ell^{1,0})$ with its counterpart $\Gamma = \iota_0 \Gamma_0 \pi_0 \in B(\ell^{1})$, etc.

\subsection{Elementary functional analytic results}

Recall the definition of the Banach space $\ell^1$ given in \eqref{eq:ell_nu_one}. 

\begin{lem}\label{l:dualbound}
The dual space $(\ell^1)^*$ is isometrically isomorphic to 
\[
\ell^\infty = \left\{ c = (c_k)_{k \ge 0} : 
\left\| c \right\|_{\infty} \bydef \max \left( |c_0|, \tfrac{1}{2} \sup_{k \geq 1} |c_k|  \right) < \infty \right\}.
\]
For all $b \in \ell^1$ and $c \in \ell^\infty$ we  have
\begin{equation}\label{e:dualbound}
\Bigl|\sum_{k\geq 0} c_k b_k \Bigr| \leq \|c\|_{\infty} \|b\|_{1}.
\end{equation}
\end{lem}

Given a sequence in $\ell^1$ we extend it symmetrically to negative indices.
The discrete convolution product~\eqref{e:convolution} then naturally works on $\ell^1$ by
\[
(b*\tilde{b})_k = \sum_{k_1+k_2 = k \atop k_1,k_2 \in \Z} b_{|k_1|} \tilde{b}_{|k_2|}.
\]
The following result states that $\ell^1$ is a Banach algebra under discrete convolutions and is useful for the analysis of nonlinear problems.

\begin{rem}\label{r:Q} 
We use the bound~\eqref{e:dualbound}
to estimate the convolution
\[
  \sup_{\|v\|_{1} \leq 1} | (b \ast v)_k | =
   \sup_{\|v\|_{1} \leq 1} \left| \sum_{k' \in \mathbb{Z}} v_{|k'|} b_{|k-k'|}   \right| \leq 
   \max \left\{ |b_k| ,\sup_{k'\geq1} \frac{|b_{|k-k'|} + b_{|k+k'|}|}{2} \right\} \bydef \Q_k(b).
\]
Given $v = (v_k)_{k \ge 0} \in \ell^1$, define $\hv \in \ell^1$ as follows:
\[
\hv_k \bydef \begin{cases} 0 & \text{if } k<m,\\
v_k & \text{if } k \geq m. \end{cases}
\]
A similar estimate as the one above leads to
\begin{equation} \label{eq:hQ}
  \sup_{\|v\|_{1} \leq 1} | (b \ast \hv)_k |  \leq 
  \sup_{k'\geq m} \frac{|b_{|k-k'|} + b_{|k+k'|}|}{2}  \bydef \hat{\Q}_k(b).
\end{equation}
Inequality \eqref{eq:hQ} is useful when computing the $Z_1$ bound (e.g. see Section~\ref{sec:Z1_CR}).
\end{rem}

\begin{lem} \label{lem:Banach_algebra_conv}
If $b, \tilde{b} \in \ell^1$, then $b \ast \tilde{b} \in \ell^1$ and 
\begin{equation} \label{eq:Banach_algebra_conv}
\| b \ast \tilde{b} \|_{1} \le \| b \|_{1} \| \tilde{b} \|_{1}.
\end{equation} 
\end{lem}

The final results of this short section concern the computation of norms of bounded linear operators defined on $\ell^1$, and are useful when computing the bounds $Z_0$ and $Z_2$.

\begin{lem}\label{l:Blnu1norm}
Let $\Gamma \in B(\ell^1)$, the space of bounded linear operators from $\ell^1$ to itself, acting as $(\Gamma b)_k =\sum_{k'\geq 0} \Gamma_{k,k'} b_{k'}$ for $k \geq 0$. Define the weights
$\omega=(\omega_k)_{k\geq0}$ by $\omega_0=1$ and $\omega_k = 2$ for $k\geq 1$. Then 
\[
   \| \Gamma \|_{B(\ell^1)} = \sup_{k' \geq 0} \frac{1}{\omega_{k'}}  \sum_{k\geq 0} | \Gamma_{k,k'} | \omega_{k} .
\] 
\end{lem}

The following consequence of Lemma~\ref{l:Blnu1norm} provides an explicit bound on norms of bounded linear operators on $\ell^1$ with a specific structure, namely as in \eqref{eq:Gamma_blocks}.

\begin{cor}\label{cor:OperatorNorm}
Let $\Gamma^{(m)}$ be an $m \times m$ matrix, 
$\{\mu_n\}_{n=m}^{\infty}$
be a sequence of numbers with 
\[
|\mu_n| \leq |\mu_{m}|, \qquad\text{for all } n \geq m,
\]
and 
$\Gamma \colon \ell^1 \to \ell^1$ be the linear operator defined by
\begin{equation} \label{eq:Gamma_blocks}
\Gamma b = 
\begin{pmatrix}
\Gamma^{(m)} &  & 0   \\
  & \mu_{m}  &   \\
0  &   & \mu_{m+1} & \\  
& & & \ddots
\end{pmatrix}
\begin{pmatrix}
 b^{(m)}  \\
  b_{m}   \\
  b_{m+1} \\
  \vdots 
\end{pmatrix}.
\end{equation}
Here $b^{(m)} = (b_0, \ldots, b_{m-1})^T \in \R^{m}$. Then $\Gamma \in B(\ell^1)$ and 
\begin{equation} \label{eq:normA}
\| \Gamma \|_{B(\ell^1)} =  \max (K, |\mu_m|),
\end{equation}
where 
\[
K \bydef \max_{0 \leq j \leq m-1} \frac{1}{\omega_j}\sum_{i=0}^{m-1} |\Gamma_{i,j}| \omega_i.
\]
\end{cor}

\subsection{The \boldmath $Y_0$ \unboldmath bound}

The nonlinear term of $F_1(\ba)$ given in \eqref{eq:F=0_CR} involves the convolution product $(\ba_1*\ba_1*\ba_1)_{k}$, which vanishes for $k \geq 3m-2$. This implies that $(F_1(\ba))_k=0$ for all $k \ge 3m-2$. Also, $(F_2(\ba))_k=0$, for all $k \ge m$. We set
\begin{align*} 
Y_0^{(1)} &\bydef \biggl| \sum_{j=1}^2 \left( A^{(m)}_{1,j} F^{(m)}_j(\ba) \right)_0 \biggr|  + 
2 \sum_{k=1}^{m-1} \biggl| \sum_{j=1}^2 \left( A^{(m)}_{1,j} F^{(m)}_j(\ba) \right)_k \biggr| 
 \\
Y_0^{(2)} &\bydef \biggl| \sum_{j=1}^2 \left( A^{(m)}_{2,j} F^{(m)}_j(\ba) \right)_0 \biggr| + 2 \sum_{k=1}^{m-1} \biggl| \sum_{j=1}^2 \left( A^{(m)}_{2,j} F^{(m)}_j(\ba) \right)_k \biggr| 
+  2 \sum_{k=m}^{3m-3} \biggl| \frac{1}{k}  (F_1(\ba))_{k} \biggr| 
\end{align*}
which is a collection of finite sums that can be evaluated with interval arithmetic. We infer that 
\[
\| [ AF(\ba)]_i \|_{1} = \biggl\| \sum_{j=1}^2 A_{i,j} F_j(\ba) \biggr\|_{1}
\le Y_0^{(i)}, \qquad\text{for } i=1,2,
\]
and we set 
\begin{equation} \label{eq:Y0_CR}
Y_0 \bydef \max \left(Y_0^{(1)},Y_0^{(2)} \right).
\end{equation}

\subsection{The \boldmath $Z_0$ \unboldmath bound}

We look for a bound of the form
$
\| I - A A^{\dagger}\|_{B(X)} \le Z_0
$.
Recalling the definitions of $A$ and $\dagA$ given in \eqref{eq:A_CR} and \eqref{eq:dagA_CR},
let $B \bydef I - A \dagA$ the bounded linear operator represented as
\begin{equation*} 
B = 
		\begin{pmatrix}
		 B_{1,1} & B_{1,2} \\
		 B_{2,1} & B_{2,2}
		\end{pmatrix}.
\end{equation*}
We remark that $( B_{i,j} )_{n_1,n_2}=0$ for any $i,j =1,2$ 
whenever $n_1 \ge m$  or $n_2 \ge m$.
Hence we can compute the norms $\|B_{i,j}\|_{B(\ell^1)}$ using Lemma~\ref{l:Blnu1norm}.
Given $h =(h_1,h_2)  \in X = \ell^1 \times \ell^1$ with $\|h\|_X = \max(\|h_1\|_{1},\|h_2\|_{1}) \leq 1$,
we obtain
\[
\|(Bh)_i \|_{1} = \biggl\|  \sum_{j=1}^2 B_{i,j} h_j \biggr\|_{1} 
\le   \sum_{i=1}^2 \| B_{i,j} \|_{B(\ell^1)}.
\]
Hence we define
\begin{equation} \label{eq:Z0_CR}
Z_0 \bydef \max\left(
 \|B_{1,1}\|_{B(\ell^1)} + \|B_{1,2}\|_{B(\ell^1)} ,
 \|B_{2,1}\|_{B(\ell^1)} + \|B_{2,2}\|_{B(\ell^1)} 
  \right),
\end{equation}
where each norm $\|B_{i,j}\|_{B(\ell^1)}$ can be computed using formula \eqref{eq:normA} with vanishing tail terms.

\subsection{The \boldmath $Z_1$ \unboldmath bound} \label{sec:Z1_CR}

Recall that we look for the bound 
$
\| A[DF(\bx) - A^{\dagger} ] \|_{B(X)} \le Z_1
$.
Given $h=(h_1,h_2) \in X$ with $\|h\|_X \le 1$, set
\[
z \bydef [DF(\ba) -  A^{\dagger} ] h.
\]
Since in $z$ some of the terms involving $((h_1)_k)_{k=0}^{m-1}$ will cancel, it is useful to introduce
$\widehat h_1$ as follows:
\[
 (\widehat h_1)_k \bydef \begin{cases} 0 & \text{if } k<m,\\
  (h_1)_k & \text{if } k \geq m. \end{cases}
\]
Then,
 \begin{align*}
 (z_1)_k &=
 \begin{cases}
 \displaystyle -3 \lambda (\ba_1*\ba_1*\widehat h_1)_k & \text{for } k=0,\dots,m-1 \\
 \displaystyle \lambda (h_1)_{k} - 3 \lambda (\ba_1*\ba_1* h_1)_k & \text{for } k \ge m \\
 \end{cases}
 \\
 (z_2)_k &=
 \begin{cases}
 	0& \text{for }k=0,\dots,m-1\\
 	 (h_2)_{k} & \text{for }k \ge m.
 \end{cases}
 \end{align*}
By \eqref{eq:hQ}, we get that 
\[
| (z_1)_k| \le 3 |\lambda| \hat{\Q}_k(\ba_1*\ba_1) , \qquad \text{for } k =0,\dots,m-1.
\]
Hence,
\begin{align*}
\|(Az)_1 \|_{1} & \le \sum_{j=1}^2 \| A_{1,j} z_j \|_{1}  
=  \sum_{k=0}^{m-1} \bigl| \bigl(A^{(m)}_{1,1} z_1^{(m)} \bigr)_k \bigr| 
+  \sum_{k \ge m}  \frac{1}{k} | (z_2)_k | \\
& \le 3 |\lambda| \sum_{k=0}^{m-1} \bigl| \bigl( |A^{(m)}_{1,1}| \hat{\Q}^{(m)}(\ba_1*\ba_1) \bigr)_k \bigr| 
+ \frac{1}{2m} \left( 2 \sum_{k \ge m} | (z_2)_k | \right) \\
& \le 3 |\lambda|\sum_{k=0}^{m-1} \bigl| \bigl( |A^{(m)}_{1,1}| \hat{\Q}^{(m)}(\ba_1*\ba_1) \bigr)_k \bigr| 
+ \frac{1}{2m} \| h_2\|_{1} \\
& \le 3 |\lambda| \sum_{k=0}^{m-1} \bigl| \bigl( |A^{(m)}_{1,1}| \hat{\Q}^{(m)}(\ba_1*\ba_1) \bigr)_k \bigr|
 + \frac{1}{2m} \,\bydef Z_1^{(1)},
\end{align*}
and similarly, now using the Banach algebra property of Lemma~\ref{lem:Banach_algebra_conv},
\begin{align*}
\|(Az)_2 \|_{1} & \le \sum_{j=1}^2 \| A_{2,j} z_j \|_{1}  
= 
 \| A_{2,1} z_1 \|_{1} 
 =  
\sum_{k=0}^{m-1} \bigl| \bigl(A^{(m)}_{2,1} z_1^{(m)} \bigr)_k \bigr| 
+  \sum_{k \ge m}  \frac{1}{k} | (z_1)_k |  \\
& \le 3 |\lambda| \sum_{k=0}^{m-1} \bigl| \bigl( |A^{(m)}_{2,1}| \hat{\Q}^{(m)}(\ba_1*\ba_1) \bigr)_k \bigr|  
+ \frac{1}{2m} \left( 2 \sum_{k \ge m} | (z_1)_k |  \right) \\
& \le 3 |\lambda|\sum_{k=0}^{m-1} \bigl| \bigl( |A^{(m)}_{2,1}| \hat{\Q}^{(m)}(\ba_1*\ba_1) \bigr)_k \bigr|   
+ \frac{|\lambda|}{2m} \left( \| h_1\|_{1} + 3 (\| \ba_1 \|_{1})^2 \| h_1\|_{1} \right) \\
& \le 3 |\lambda| \sum_{k=0}^{m-1} \bigl| \bigl( |A^{(m)}_{2,1}| \hat{\Q}^{(m)}(\ba_1*\ba_1) \bigr)_k \bigr|  
 + \frac{|\lambda|}{2m} \left( 1 + 3 (\| \ba_1 \|_{1})^2  \right) \,\bydef Z_1^{(2)}.
\end{align*}
We thus define 
\begin{equation} \label{eq:Z1_CR}
Z_1 \bydef \max\left(Z_1^{(1)},Z_1^{(2)} \right).
\end{equation}

\subsubsection{The \boldmath $Z_2$ \unboldmath bound}

Let $r>0$ and $c=(c_1,c_2) \in B_r(\ba)$, that is $\| c- \ba \|_{X} = \max( \|c_1- \ba_1\|_{1},\|c_2- \ba_2\|_{1} ) \le r$. Given $\| h \|_X \leq 1$, note that $\left( [DF_2(c)-DF_2(\ba)] h \right)_k=0$ and that
\[
\left( [DF_1(c)-DF_1(\ba)] h \right)_k = 
- 3 \lambda \left( (c_1 * c_1 - \ba_1 * \ba_1) * h_1 \right)_{k} 
\]
so that 
\begin{align*}
\| A [DF(c) - DF(\ba)]\|_{B(X)} &= \sup_{\| h \|_X \leq 1} \| A [DF(c) - DF(\ba)] h \|_{X} \\
& \le \| A \|_{B(X)} \sup_{\| h \|_X \leq 1} \| [DF(c) - DF(\ba)] h \|_{X} \\
& = 3 |\lambda| \| A \|_{B(X)} \sup_{\| h \|_X \leq 1} \| (c_1- \ba_1)*(c_1+\ba_1) * h_1 \|_{1} \\
& \le 3 |\lambda| \| A \|_{B(X)} \sup_{\| h \|_X \leq 1} \| c_1- \ba_1 \|_{1} \| c_1+ \ba_1 \|_{1} \| h_1 \|_{1}  \\
& \le 3 |\lambda| \| A \|_{B(X)} r (\| c_1 \|_{1}+ \| \ba_1 \|_{1}) \\
& \le 3 |\lambda| \| A \|_{B(X)} r (r+ 2\| \ba_1 \|_{1}).
\end{align*}
Then, assuming a loose a priori bound $r \le 1$ on the radius, we set
\begin{align} \label{eq:Z2_CR}
Z_2 &\bydef 3 |\lambda| \| A \|_{B(X)} (1+ 2\| \ba_1 \|_{1}),
\end{align}
with
\[
 \| A \|_{B(X)} =  \max\left(
 \|A_{1,1}\|_{B(\ell^1)} + \|A_{1,2}\|_{B(\ell^1)} ,
 \|A_{2,1}\|_{B(\ell^1)} + \|A_{2,2}\|_{B(\ell^1)} 
  \right),
\]
where each operator norm $\|A_{i,j}\|_{B(\ell^1)}$ can be computed using formula \eqref{eq:normA}. 


\section{The travelling wave problem} \label{s:TW}

Before we discuss the existence results for the travelling wave problem~\eqref{e:TW} we discuss its spectral properties.

\subsection{Spectral properties}
\label{s:TWspectralflow}

Equation~\eqref{e:TW} may be written as a system of first order equations
\begin{equation}\label{e:TWsystem}
\left\{
\begin{aligned}
  u_t  &= v,\\
  v_t  &= cv - u_{x_1 x_1} -u_{x_2 x_2} - \psi_\lambda(u), 
  \end{aligned}
  \right.
\end{equation}
with Neumann boundary conditions on the square.
The spectral problem for~\eqref{e:TWsystem} 
is directly related to the spectral problem for the parabolic equation
\begin{equation}\label{e:parabolic}
  u_t = u_{x_1 x_1} + u_{x_2 x_2} + \psi_\lambda(u),
\end{equation}
again with Neumann boundary conditions on the square.
First, we note that any equilibrium of~\eqref{e:TW} is of the form $(u,v)=(u_*,0)$, with $u_*$ an equilibrium of~\eqref{e:parabolic}.
Furthermore, the eigenvalue problems of the linearized operators at these equilibria are 
\begin{equation}\label{e:rho}
\left\{
\begin{aligned}
  \rho u &= v,\\
  \rho v &= cv -  u_{x_1 x_1}-u_{x_2 x_2} - \psi'_\lambda(u_*) u, 
  \end{aligned}
  \right.
\end{equation}
and
\begin{equation}\label{e:sigma}
    \sigma u  = u_{x_1 x_1} + u_{x_2 x_2} + \psi'_\lambda(u_*) u,	
\end{equation}
respectively, both with Neumann boundary conditions.
Hence eigenvalues $\rho$ of~\eqref{e:rho} and eigenvalues $\sigma$ of~\eqref{e:sigma} are related through
\begin{equation}\label{e:sigmamu}
  \sigma = c \rho - \rho^2.
\end{equation}
Since the elliptic operator in~\eqref{e:sigma} is self-adjoint, all eigenvalues $\sigma$ are real.
Each negative eigenvalue $\sigma$ of~\eqref{e:sigma}, of which there are infinitely many, corresponds to a pair of eigenvalues 
\[
  \rho = \rho_{\pm}(\sigma) = \frac{c}{2} \pm \biggl(\frac{c^2}{4} - \sigma \biggr)^{\!1/2} ,
\]
one positive and one negative (which is of course consistent with~\eqref{e:TWsystem} being strongly indefinite). 
For each positive $\sigma$, of which there are at most finitely many, there are two eigenvalues $\rho=\rho_\pm$
(a double eigenvalue for $\sigma=\frac{c^2}{4}$), both with positive real part. 
In particular, all eigenvalues of~\eqref{e:rho} lie in the union $\{ \text{Im}(z)=0 \} \cup \{ \text{Re}(z)=\frac{c}{2} \} \subset \mathbb{C}$.
 Hence, when parameters are varied eigenvalues can only pass from the left half-plane to the right half-plane through the origin.
Since the eigenvalue problem~\eqref{e:rho} associated to~\eqref{e:TWsystem} is thus directly and explicitly related to the self-adjoint eigenvalue problem~\eqref{e:sigma}, it is reasonable to expect that 
for $c>0$ (or $c<0$) the spectral flow for the linearization of~\eqref{e:TWsystem} is well-defined, and this is indeed the case, see~\cite[Section~5.2]{BakkervdBergvdVorst} and~\cite{MR1331677}. Furthermore, it follows from~\eqref{e:sigmamu} that along a homotopy eigenvalues $\rho$ of~\eqref{e:rho} and $\sigma$ of~\eqref{e:sigma} cross the origin simultaneously and in the same direction. 
  The spectral flows for~\eqref{e:TWsystem} and~\eqref{e:parabolic} are thus ``the same'' in the sense that the relative index of a pair of equilibria for~\eqref{e:TWsystem} is equal to the relative index of this pair for~\eqref{e:parabolic}. Since the latter is easier to analyse (it is scalar), we compute relative indices using the parabolic equation~\eqref{e:parabolic} and then draw conclusions for the strongly indefinite system~\eqref{e:TWsystem}, or, equivalently, the travelling wave problem~\eqref{e:TW}.

\subsection{Problem reformulation}

As explained in Section~\ref{s:TWspectralflow}, to draw conclusions about~\eqref{e:TW},
we compute equilibria and associated Morse indices
of the parabolic equation
\begin{equation} \label{eq:nonlinear_parabolic}
  u_t = \Delta u +  \psi_\lambda(u) = \Delta u + \lambda (u-u^3),
\end{equation}
with Neumann boundary conditions on the square $[0,\pi] \times [0,\pi]$.
Here we have chosen $\lambda_1=\lambda_2 \bydef \lambda$.
We perform the cosine transform
\begin{equation*}
  u(x) = \sum_{k \in \Z^2} a_k e^{i k \cdot x}
  = \sum_{k \in \N^2} m_k a_k \cos(k_1 x_1) \cos(k_2 x_2)  
\end{equation*}
where the multiplicities are
\[
  m_k=m_{k_1,k_2} \bydef \begin{cases}
  1 & \text{for } k_1=k_2=0\\
  2 & \text{for } k_1=0, k_2>0\\
  2 & \text{for } k_1>0, k_2=0\\
  4 & \text{for } k_1>0, k_2>0.
  \end{cases}
\]
We will from now on assume $a_{k_1,k_2} = a_{|k_1|,|k_2|} \in \R$.
Since the Neumann boundary conditions allow a smooth doubly $2\pi$-periodic extension of equilibrium solutions of~\eqref{eq:nonlinear_parabolic} to $\R^2$ with the symmetries $u(x_1,x_2)=u(-x_1,x_2)=u(x_1,-x_2)$, assuming a cosine-cosine expansion does not impose any restrictions.
The equilibrium equations for the unknowns $(a_k)_{k\in\N^2}$ become 
\begin{equation} \label{e:eqa}
F_k(a) \bydef  m_k \big[ (- (k_1^2+k_2^2) +\lambda) a_k - \lambda (a*a*a)_k \bigl],
\end{equation}
with the usual convolution. 
Here the choice to include the factor $m_k$
is for the same reason as the factor $2$ in~\eqref{eq:F=0_CR}: it makes the symmetry of the Jacobian $DF$ apparent.
We denote $F(a) =  \{F_k(a)\}_{k \in \N^2}$.
For the norm in Fourier space we select the 
$\ell^1$-norm:
\begin{equation}\label{e:ell1nu2D}
  \|a\|_{1} \bydef \sum_{k \in \N^2} m_k |a_k| ,
\end{equation}
with 
$ |k| \bydef  \max\{|k_1|,|k_2|\}$.
One nice thing about the $\ell^1$-norm~\eqref{e:ell1nu2D} is the Banach algebra property
$  \|a*b\|_{1} \leq \|a\|_{1} \|b\|_{1}$.
This makes our space 
\begin{equation}
	\label{e:X1}
  X = \{ a=(a_k)_{k\in\N^2} \,:\, a_k \in \R \,,\,   \|a\|_{1} < \infty \}
\end{equation}
into a Banach algebra. 

Computing equilibria of \eqref{eq:nonlinear_parabolic} reduces to find $a \in X$ such that $F(a)=0$, where $F$ is given component-wise by \eqref{e:eqa}. The Newton-Kantorovich approach of Theorem~\ref{thm:radii_polynomials} is applied to achieve this task. Following a similar approach as in Section~\ref{s:CR}, we compute an approximate solution $\ba$ of $F=0$, define the linear operators $\dagA$ and $A$, and compute the bounds $Y_0$, $Z_0$, $Z_1$ and $Z_2(r)$ satisfying \eqref{eq:general_Y_0}, \eqref{eq:general_Z_0}, \eqref{eq:general_Z_1} and \eqref{eq:general_Z_2}. 
The derivation of the detailed expressions for $Y_0$, $Z_0$, $Z_1$ and  $Z_2(r)$ is omitted, as this analysis is analogous to Section~\ref{s:CR}, see also~\cite{OK2D} for a similar (but more involved) problem in two space dimensions.
Defining the radii polynomial $p(r)$ as in \eqref{eq:general_radii_polynomial}, if there is $r_0>0$ such that $p(r_0)<0$, then there exists a unique $\ta$ with $\| \ta - \ba\| < r_0$ such that $F(\ta) = 0$. After having obtained computer-assisted proofs of a solution $\ta$ of $F=0$, we use the theory of Section~\ref{sec:rig_comp_spectrum} and Section~\ref{sec:rig_comp_relative_indices} to compute their relative indices. Using this approach, we proved the solutions and relative indices depicted in Figure~\ref{f:TW}. All proofs used truncation dimension $m=20$ in Fourier space.
The code {\tt proveall12.m}, available at~\cite{codes_webpage}, performs all the computations with interval arithmetic.  

Having established these equilibria and their relative indices, the result on the existence and multiplicity of travelling waves for the PDE problem~\eqref{e:TW}, as formulated in Theorem~\ref{t:TW}, follow from the forcing lemma~\ref{l:forcing}.


\section{The Ohta-Kawasaki problem} \label{s:OK}

Our third example is the Ohta-Kawasaki equation \eqref{e:OK}, which we recall is given by
\[
	\begin{cases}
   u_t = - u_{xxxx} - (\psi_\lambda(u))_{xx} - \lambda \sigma u, &\quad\text{for } x\in [0,\pi],\\
    u_x(t,0)=u_x(t,\pi)=0, \\
    u_{xxx}(t,0)=u_{xxx}(t,\pi)=0,\\
	\int_0^\pi u(0,x) dx =0 .
  \end{cases}
\]
Plugging the cosine Fourier expansion (using the symmetry coming from the Neumann boundary conditions and the fact that $\int_0^\pi u(0,x) dx =0$)
\[
u(x)=\sum_{k \in \Z} a_k e^{ik x}, \qquad \text{with } a_k \in \R,~~ a_{-k}=a_k,~~\text{and}~~ a_0=0,
\]
into the steady state equation $- u_{xxxx} - (\psi_\lambda(u))_{xx} - \lambda \sigma u = 0$ yields
\begin{equation} \label{eq:F_k_OK}
F_k(a) \bydef  \left( -k^4 + \lambda k^2  - \lambda \sigma \right)  a_k - \lambda k^2 (a^3)_k=0.
\end{equation}
Here $a=(a_k)_{k \ge 1}$ and
\[
(a^3)_k \bydef \sum_{\stackrel{k_1+k_2+k_3=k}{k_i \in \Z \setminus \{0\} }} a_{|k_1|} a_{|k_2|} a_{|k_3|}.
\]
The relations $F_{-k}=F_k$ and $F_0(a)=0$ imply that we only need to solve $F_k=0$ for $k \ge 1$.
We thus set $F \bydef (F_k)_{k \ge 1}$. The Banach space $X$ used in the present example is
\begin{equation}\label{def:ellOne}
X = \left\{ a = \{a_k\}_{k \ge 1 }  :  \| a \|_{1} \bydef 2 \sum_{k \ge 1} |a_k|< \infty \right\}.
\end{equation}

Computing equilibria of \eqref{e:OK} reduces to find $a \in X$ such that $F(a)=0$, where $F$ is given component-wise by \eqref{eq:F_k_OK}. This is done by applying Theorem~\ref{thm:radii_polynomials}. As in the other examples, we compute an approximate solution $\ba$ of $F=0$, define the linear operators $\dagA$ and $A$, and compute the bounds $Y_0$, $Z_0$, $Z_1$ and $Z_2(r)$ satisfying \eqref{eq:general_Y_0}, \eqref{eq:general_Z_0}, \eqref{eq:general_Z_1} and \eqref{eq:general_Z_2}. We omit the derivation of these bounds. Defining the radii polynomial $p(r)$ as in \eqref{eq:general_radii_polynomial}, if there is $r_0>0$ such that $p(r_0)<0$, then there exists a unique $\ta$ with $\| \ta - \ba\| < r_0$ such that $F(\ta) = 0$. After having obtained computer-assisted proofs of a solution $\ta$ of $F=0$, we use the theory of Section~\ref{sec:rig_comp_spectrum} and Section~\ref{sec:rig_comp_relative_indices} to compute their relative indices. Using this approach, we proved the solutions and relative indices depicted in Figure~\ref{f:OK}. All proofs used truncation dimension $m=40$ in Fourier space. The code {\tt script\_proofs\_OK.m}, available at~\cite{codes_webpage}, performs all the computations with interval arithmetic. 

Once we have the equilibria and their relative indices, the result on the existence and multiplicity of connection orbits for the PDE problem~\eqref{e:OK}, as stated in Theorem~\ref{t:OK}, follow from the forcing lemma~\ref{l:forcing}, see the proof strategy outlined in the introduction.


\end{document}